\documentclass[11pt]{article}
\usepackage{amssymb,latexsym,amsmath,amsbsy,amsthm,amsxtra,amsgen,dsfont,pdfsync,graphicx,bm,mathtools}
\usepackage{extpfeil}
\usepackage{mathrsfs}
\usepackage{bbm}
\usepackage{indentfirst}
\usepackage{cite}
\usepackage{mathrsfs} 
\usepackage{color}
\usepackage{url}
\usepackage{authblk}
\usepackage{blindtext}
\usepackage{booktabs}
\usepackage{titling}
\predate{}
\postdate{}

\newfont{\msbm}{msbm10 at 11pt}

\newfont{\msbmsm}{msbm10 at 8pt}

\newtheorem{Theo}{Theorem}
\newtheorem{Lemma}[Theo]{Lemma}

\newtheorem{Prop}[Theo]{Proposition}

\newcommand{\qedwhite}{\hfill \ensuremath{\Box}}

\begin{document}

\title{A Yaglom type asymptotic result for subcritical branching Brownian motion with absorption}
\author{Jiaqi Liu\thanks{{\it E-mail address:} jil1131@ucsd.edu}\,}
\affil{\fontsize{9}{10.8}\itshape{Department of Mathematics, University of California San Diego, La Jolla, CA 92093-0112, USA}}
\date{}

\maketitle
\begin{abstract}
We consider a slightly subcritical branching Brownian motion with absorption, where particles move as Brownian motions with drift $-\sqrt{2+2\varepsilon}$, undergo dyadic fission at rate $1$, and are killed when they reach the origin. We obtain a Yaglom type asymptotic result, showing that the long run expected number of particles conditioned on survival grows exponentially as $1/\sqrt{\varepsilon}$ as the process approaches criticality.
\end{abstract}

{\small MSC: Primary 60J80; secondary 60J65; 60J25

}

\section{Introduction}
We consider a branching Brownian motion (BBM) with absorption, in which particles move as Brownian motions with drift $-\rho$ ($\rho \in \mathbb{R}$), undergo dyadic branching at rate $1$, and are killed when they reach the origin. Kesten \cite{kesten1978branching} first studied this model in 1978 and showed that  when $\rho \geq \sqrt{2}$, BBM with absorption dies out almost surely while when $\rho < \sqrt{2}$, there is a positive probability of survival.  Therefore, $\rho=\sqrt{2}$ is the critical value separating the supercritical case $\rho<\sqrt{2}$ and the subcritical case $\rho>\sqrt{2}$. 

There has been a long-standing interest in problems related to the asymptotic behavior of the survival probability. After introducing the model, Kesten \cite{kesten1978branching} obtained upper and lower bounds on the probability in the critical case that the process survives until some large time. Kesten's result was further improved by Berestycki et al.\cite{berestycki2014critical}. The asymptotic result for the survival probability in the supercritical case was obtained by Harris et al. \cite{harris2006further} through studying the FKPP equation associated with this process. Derrida and Simon \cite{derrida2007survival} gave a quite precise prediction for the survival probability in the slightly supercritical case  through nonrigorous PDE methods, where the drift $\rho$ is slightly below the critical value. Rigorous probabilistic proofs were provided by \cite{berestycki2011survival}.  In this paper, we are interested in a nearly critical case, where $\rho$ approaches the critical value $\sqrt{2}$ from above. For notational simplicity, we write $\rho^2/2-1=\varepsilon$ where $0<\varepsilon<1$ and $\varepsilon$ approaches to $0$. We denote by $P^x_{-\rho}$ the probability measure for branching Brownian motion started from a single particle at $x>0$ with drift $-\rho$ and absorbed at $0$, and by $\mathcal{N}_t^{-\rho}$ the set of surviving particles under $P^{x}_{-\rho}$ at time $t$. The set of positions of particles at time $t$ under $P^x_{-\rho}$ is $\{Y_u(t): u\in \mathcal{N}_t^{-\rho}\}$ and the number of particles at time $t$ under $P^x_{-\rho}$ is $N_t^{-\rho}$. 

In the subcritical case, almost surely, the process becomes extinct. However, it is interesting to consider the behavior of the process conditioned on survival up to a large time. This type of result is called a Yaglom theorem and has been considered by Yaglom\cite{yaglom1947certain} in ordinary branching processes. A similar question was studied by Chauvin and Rouault \cite{chauvin1988kpp} in the setting of BBM without absorption. Let $P$ be the law of an ordinary BBM started from $0$ without drift and absorption and $\mathcal{N}_t$ be the set of particles at time $t$. Chauvin and Rouault first gave an asymptotic expression for the probability of existence of particles to the right of $\rho t+x$ at some large time, $P(\exists u\in \mathcal{N}_t: Y_u(t)>\rho t+x)$. Then they obtained a limit distribution for the number of particles that drift above $\rho t+x$ at time $t$ conditioned on the presence of such particles for $\rho>\sqrt{2}$. Harris and Harris \cite{harris2007survival} obtained related results for BBM with absorption. They derived a large-time asymptotic formula for the survival probability in the subcritical case. They proved that for $\rho >\sqrt{2}$ and $x>0$, there exists a constant $K_\varepsilon$ that is independent of $x$ but dependent on the drift $\rho$, and therefore on $\varepsilon$, such that,
\begin{equation}\label{Harris}
\lim_{t\rightarrow \infty}P^x_{-\rho}(N_t^{-\rho}>0)\frac{\sqrt{2\pi t^3}}{x}e^{-\rho x+\varepsilon t}=K_\varepsilon,
\end{equation}
 and furthermore,
\begin{equation}\label{HarrisE}
\lim_{t\rightarrow \infty}\frac{P^x_{-\rho}(N_t^{-\rho}>0)}{E^x_{-\rho}[N_t^{-\rho}]} = \frac{1}{2}\rho^2K_\varepsilon.
\end{equation}
Comparing this with Chauvin and Rouault's result, as $t$ goes to infinity, $P^x_{-\rho}(N_t^{-\rho}>0)$ and $P(\exists u\in \mathcal{N}_t: Y_u(t)>\rho t+x)$ are the same on the exponential scale but different in terms of the polynomial corrections. The constant $K_\varepsilon$ plays an important role in calculating the limiting expected number of particles alive conditioned on at least one surviving. In fact, it is pointed out by Harris and Harris in \cite{harris2007survival} that as a direct consequence of (\ref{HarrisE}), we have
\begin{equation}\label{Chauvin}
\lim_{t\rightarrow\infty}E^{x}_{-\rho}[N_{t}^{-\rho}|N_t^{-\rho}>0]=\frac{2}{\rho^2 K_{\varepsilon}}.
\end{equation}
Furthermore, by using the method of Chauvin and Rouault \cite{chauvin1988kpp}, it follows from (\ref{HarrisE}) that there is a probability distribution $(\pi_j)_{j\geq1}$ such that
\[
\lim_{t\rightarrow\infty}P^{x}_{-\rho}(N_{t}^{-\rho}=j|N_t^{-\rho}>0)= \pi_j.
\]
Our main result, which is Theorem 1 below, analyzes the asymptotic behavior of ($\ref{Chauvin}$) as $\varepsilon$ goes to $0$. We show that the long-run expected number of particles conditioned on survival grows exponentially as the process gets closer to being critical.
\begin{Theo}\label{main}
 There exist positive constants $C_1$ and $C_2$ such that for $\varepsilon$ small enough,
\begin{equation}
e^{C_{1}/\sqrt{\varepsilon}} \leq \lim_{t\rightarrow \infty}E^{x}_{-\rho}[N_{t}^{-\rho}|N_t^{-\rho}>0]\leq e^{C_{2}/\sqrt{\varepsilon}}.
\label{goal}
\end{equation}
\end{Theo}

Kesten \cite{kesten1978branching} had a result of this type in the critical case. Recently, Maillard and Schweinsberg \cite{berestycki2020yaglom} proved Yaglom-type limit theorems for more specific behaviors of the process in the critical case. They derived the asymptotic distributions of the survival time, the number of particles and the position of the rightmost particle for the process conditioned on survival for a long time. In the setting of supercritical branching random walk (BRW), Gantert et al. \cite{gantert2011asymptotics} and Pain \cite{pain2018near} considered problems with similar flavor. Let $\gamma$ be the asymptotic speed of the rightmost position in the BRW. Gantert et al. \cite{gantert2011asymptotics} studied the probability that there exists an infinite ray which stays above the line of slope $\gamma-\varepsilon$ as $\varepsilon$ goes to $0$. Having an infinite ray staying above the line with slope $\gamma-\varepsilon$ can be viewed as survival with slightly supercritical drift. They proved that when $\varepsilon\rightarrow0$, this probability decays as $\exp(-(c+o(1))/\sqrt{\varepsilon})$ where $c$ is a positive constant depending on the distribution of the branching random walk. In \cite{pain2018near}, Pain studied the near-critical Gibbs measure and the partition function of parameter $\beta$ on the $n-$th generation of the BRW. In his setting, the inverse temperature $\beta$ is a function of $n$ and approaches to the critical value $1$ both from above and below. Our setting can be viewed as a iterated limit where we first let time $t$ go to infinity and then let the process approach to criticality, while Pain's setting can be viewed as a double limit where the process approaches to criticality at the same time when the generation goes to infinity.

It is important to point out that Theorem \ref{main} does not imply that as the process approaches criticality,  we have $\log N_t^{-\rho}=O(\varepsilon^{-1/2})$ conditioned on survival up to time $t$ in a typical realization of the process. We conjecture that there is a big difference between the expected number and the typical number of surviving particles because the expectation is dominated by rare events where an unusually large number of particles survive. We further conjecture that for $\varepsilon$ sufficiently small, the logarithm of the number of particles at time $t$ conditioned on survival up to time $t$ is typically around $\varepsilon^{-1/3}$. 

The proof of Theorem \ref{main} relies on a better understanding of $K_\varepsilon$ as the drift approaches the critical value. According to (\ref{Harris}), studying $K_\varepsilon$ boils down to finding an asymptotic expression for the survival probability in the slightly subcritical regime. Here we  apply a spinal decomposition to transform survival probability to expectation of the reciprocal of a martingale.

As in Harris and Harris \cite{harris2007survival}, define
\[
V(t):=\sum_{u\in \mathcal{N}_t^{-\rho}}Y_u(t)e^{\rho Y_u(t)+\varepsilon t}.
\]
\noindent Lemma 2 in \cite{harris2007survival} shows that $\{V(t)\}_{t\geq 0}$ a martingale under $P^{x}_{-\rho}$. We can define a new measure $Q^x$ on the same probability space as $P^x_{-\rho}$ via $\{V(s)\}_{s\geq 0}$,
\begin{equation}\label{Q}
\frac{dQ^x}{dP^x_{-\rho}}\Big|_{\mathcal{F}_s}=\frac{V(s)}{V(0)}.
\end{equation}
Under the measure $Q^x$, there is one chosen particle which is called the spine whose law is altered and all subtrees branching off the spine behave like the original branching Brownian motion with absorption. The spine moves as a Bessel-3 process starting from $x$. With accelerated rate 2, the initial ancestor undergoes binary fission. The spine is chosen uniformly from the two offspring, and the remaining offspring initiates an independent copy of the original branching Brownian motion with absorption. In this paper $Q^x$ is used both for probability and expectation. Representing $K_\varepsilon$ under $Q^x$ in (\ref{Harris}),
\[
K_{\varepsilon}=\lim_{t\rightarrow \infty}Q^x\bigg[\frac{V(0)}{V(t)};N_t^{-\rho}>0\bigg]\frac{\sqrt{2\pi t^3}}{x}e^{-\rho x+\varepsilon t}=\lim_{t\rightarrow \infty}\sqrt{2\pi t^3}Q^x\bigg[\frac{1}{\sum_{u\in \mathcal{N}_t^{-\rho}}Y_u(t)e^{\rho Y_u(t)}}\bigg].
\]
As a result, Theorem \ref{main} follows from the following proposition.
\begin{Prop}\label{K}
 There exist positive constants $C_1$ and $C_2$ such that for $\varepsilon$ small enough,
 \begin{equation}\label{limsup}
\limsup_{t\rightarrow \infty}\sqrt{2\pi t^3}Q^{x}\bigg[\frac{1}{\sum_{u\in\mathcal{N}_t^{-\rho}}Y_u(t)e^{\rho Y_u(t)}}\bigg]\leq e^{-C_{1}/\sqrt{\varepsilon}},
\end{equation}
\begin{equation}\label{liminf}
\liminf_{t\rightarrow \infty}\sqrt{2\pi t^3}Q^{x}\bigg[\frac{1}{\sum_{u\in \mathcal{N}_t^{-\rho}}Y_u(t)e^{\rho Y_u(t)}}\bigg]\geq e^{-C_{2}/\sqrt{\varepsilon}}.
\end{equation}
\end{Prop}
We point out here that we cannot specify choices of $C_1$ and $C_2$. We are only able to determine upper and lower bounds for $\lim_{t\rightarrow \infty}E^{x}_{-\rho}[N_{t}^{-\rho}|N_t^{-\rho}>0]$. 

The rest of this paper is organized as follows. In Section 2, results related to Brownian motion and the Bessel-3 process will be summarized. Sections 3 and 4 will be devoted to the proofs of the upper bound and lower bound in Proposition \ref{K} respectively. Throughout this paper,  for two nonzero functions $f(t)$ and $g(t)$, we use the notation $f(t)\sim g(t)$ as $t\rightarrow a$ to mean that $\lim_{t\rightarrow a}f(t)/g(t)=1$. We summarize some of the notation that are used throughout the paper in Table \ref{table}.
\begin{table}[!ht]  
  \centering  
   \begin{tabular}{c|p{12cm}}
       \toprule
    $P^x_{-\rho}$ & the probability measure of the branching Brownian motion started from a single particle at $x>0$ with drift $-\rho$ and absorbed at $0$.  \\
    $\mathcal{N}_t^{-\rho}$   & the set of surviving paticles under $P^x_{-\rho}$ at time $t$. \\
    $Q^x$ & the probability measure on the same probability space as $P^x_{-\rho}$ defined via a spine change of measure.   \\
    $\{\xi_t\}_{t\geq 0}$ & the trajectory of the spine.\\
    $\{\zeta_{s}\}_{0\leq s\leq t}$ & the reversed trajectory of the spine up to time $t$, i.e. $\{\xi_{t-s}\}_{0\leq s\leq t}$. \\
    $Q^{x,t,z}$ & the probability measure of the branching process under $Q^x$ whose spine starts from $x$ and is conditioned to end up at $z$ at time $t$, i.e. $Q^x(\cdot|\xi_t=z)$ or $Q^x(\cdot|\zeta_0=z)$.    \\
    $\{B_t^{x,u,y}\}_{0\leq t\leq u}$ & a Brownian bridge from $x$ to $y$ over time $u$.\\
    $\{X_t^{x,u,y}\}_{0\leq t\leq u}$ & a Bessel bridge from $x$ to $y$ over time $u$. If clear from the context, we will write $\{X_r\}_{0\leq r\leq u}$ for simplicity. \\
    $p_t^{x,u,y}(\cdot)$ & the transition density of a Bessel process from $x$ to $y$ over time $u$ at time $t$.\\
    $\{R_r^z\}_{r\geq 0}$ & a Bessel-3 process started from $z$. \\
    $p_t(x,\cdot)$ & the transition density of a Bessel process started from $x$ at time $t$. \\ \bottomrule
  \end{tabular}
  \caption {\textit{Index of some of the notation}} \label{table}
\end{table}

\section{Preliminary results}

In this section, we will summarize results pertaining to Brownian motion and the Bessel-3 process which will be used later in the proof. For further properties of the Brownian motion and the Bessel process, we refer the reader to Borodin and Salminen \cite{borodin2012handbook}. 

Let$\{B_t\}_{t\geq 0}$ be standard Brownian motion and $\{B^{x,u,y}_{t}\}_{0\leq t\leq u}$ be a Brownian bridge from $x$ to $y$ over time $u$. Standard Brownian bridge refers to the Brownian bridge from 0 to 0 in time 1, $\{B^{0,1,0}_{t}\}_{0\leq t\leq 1}$. Reflected Brownian bridge is the absolute value of the Brownian bridge, $\{|B_t^{x,u,y}|\}_{0\leq t\leq u}$. Now we will be able to state the following lemma. Lemma \ref{RBMa=0} derives the limit of the probability that a reflected standard Brownian bridge always stays below a line $at+b$ as $b(a+b)$ approaches 0.  We will prove by first obtaining the explicit probability formula written as an infinite sum and then analyzing its limiting behavior through Jacobi theta functions.

\begin{Lemma}\label{RBMa=0} 
For $a\geq 0$ and $b>0$, we have
\begin{equation*}
P\Big(\sup_{0\leq t\leq 1}\Big(|B^{0,1,0}_t|-at\Big)<b\Big)\sim \sqrt{\frac{2\pi}{b(a+b)}}e^{-\frac{\pi^2}{8b(a+b)}}\quad \text{as} \;\; b(a+b)\downarrow 0.
\end{equation*}
\end{Lemma}
\noindent\textit{Proof.} 
According to Theorem 7 in  \cite{salminen2011hitting}, we have 
\begin{equation}\label{RBMa=0series}
P\Big(\sup_{0\leq t\leq 1}\Big(|B_t^{0,1,0}|-at\Big)<b\Big)=\sum_{k=-\infty}^{\infty}(-1)^ke^{-2k^2b(a+b)}.
\end{equation}
To have a better understanding of this expression for small values of $b(a+b)$, we need to introduce the Jacobi theta functions of type 2, $\vartheta_2(z|\tau)$ and type 4, $\vartheta_4(z|\tau)$ and their relationship. A good reference would be Section 16 of \cite{abramowitz1948handbook}. We have
\[
\vartheta_2(z|\tau):=2e^{i\pi\tau/4}\sum_{k=0}^{\infty}e^{i\pi\tau k(k+1)}\cos((2k+1)z),
\]
\[
\vartheta_4(z|\tau):=\sum_{k=-\infty}^{\infty}(-1)^ke^{i\pi\tau k^2}e^{2kiz}.
\]
As a special case of Jacobi's imaginary transformation,
\[
\vartheta_4(0|\tau)=(-i\tau)^{-1/2}\vartheta_2\Big(0|-\frac{1}{\tau}\Big).
\]
Then (\ref{RBMa=0series}) can be written in terms of Jacobi theta functions,
\[
P\Big(\sup_{0\leq t\leq1}\Big(|B^{0,1,0}_t|-at\Big)<b\Big)=\vartheta_4\Big(0|\frac{2b(a+b) i}{\pi}\Big)=\sqrt{\frac{\pi}{2b(a+b)}}\vartheta_2\Big(0|\frac{\pi i}{2b(a+b)}\Big).\]
We want to explore the limiting behavior of $P(\sup_{0\leq t\leq1}|B^{0,1,0}_t|-at<b)$ as $b(a+b)$ approaches 0.
By the series representation for the theta function $\vartheta_2$, if $e^{i\pi\tau}\in\mathbb{R}$  and $e^{i\pi\tau}\rightarrow 0$, then
\[
\vartheta_2(0|\tau)\sim 2e^{i\pi\tau/4} .
\]
Therefore, as $b(a+b)$ approaches $0$ from above,
\[
P\Big(\sup_{0\leq t\leq 1}\Big(|B_t^{0,1,0}|-at\Big)<b\Big)=\sqrt{\frac{\pi}{2b(a+b)}}\vartheta_2\Big(0|\frac{\pi i}{2b(a+b)}\Big)\sim \sqrt{\frac{2\pi}{b(a+b)}}e^{-\frac{\pi^2}{8b(a+b)}}.
\]
\qedwhite

Below we will present a stochastic dominance relation between Brownian bridges with the same length but different endpoints.

\begin{Lemma}\label{stochasticdominanceBM}
For every $t>0$, if $x_1\geq x_2$ and $y_1\geq y_2$, then $\{B^{x_1,t,y_1}_r\}_{0\leq r\leq t}$ stochastically dominates $\{B^{x_2,t,y_2}_r\}_{0\leq r\leq t}$. In other words, these two processes can be constructed on some probability space such that almost surely for all $r\in [0,t]$, 
\[
B^{x_1,t,y_1}_r\geq B^{x_2,t,y_2}_r.
\]
\end{Lemma}
\begin{proof}
According to IV.21 of \cite{borodin2012handbook}, after some computations,  $\{B^{x_1,t,y_1}_r\}_{0\leq r \leq t}$ and $\{B^{x_2,t,y_2}_r\}_{0\leq r\leq t}$ can be expressed in terms of $\{B^{0,t,0}_r\}_{0\leq r\leq t}$,
\[
B^{x_i,t,y_i}_r=\frac{t-r}{t}x_i+\frac{r}{t}y_i+B^{0,t,0}_r, \quad\text{for}\; 0\leq r\leq t \;\text{and}\; i=1,2.
\]
Since $x_1\geq x_2$ and $y_1\geq y_2$, $\{B^{x_1,t,y_1}_r\}_{0\leq r\leq t}$ stochastically dominates $\{B^{x_2,t,y_2}_r\}_{0\leq r\leq t}$ from the above coupling.
\end{proof}

Next we are going to introduce results pertaining to the Bessel-3 process. The Bessel-3 process is defined to be the radial part of a three-dimensional Brownian motion. Since only the Bessel-3 process will be considered in this paper, below we will write the Bessel process for convenience. Also, the Bessel process is identical in law to a one dimensional Brownian motion conditioned to avoid the origin.  Let $p_t(x,y)$ be the transition density of a Bessel process started from $x$ at time $t$. We have 
\[
p_t(x,y)=\frac{y}{x}\frac{1}{\sqrt{2\pi t}} e^{-\frac{(y-x)^2}{2t}}\Big(1-e^{-2xy/t}\Big).
\]
Similarly to the Brownian motion setting, define $\{X^{x,u,y}_t\}_{0\leq t\leq u}$ as a Bessel bridge from $x$ to $y$ over time $u$ and $p_t^{x,u,y}(z)$  as its transition density at time $t$.  Specifically, $\{X^{0,1,0}_t\}_{0\leq t\leq 1}$ is a Brownian excursion. It is shown in the proof of Lemma 7 in \cite{harris2007survival} that a Bessel bridge is identical in law to a Brownian bridge that is conditioned to avoid the origin. Since a time-reversed Brownian bridge is also a Brownian bridge,  we see that a time-reversed Bessel bridge is also a Bessel bridge. To be more precise,
\[
\{X^{x,t,z}_{t-s}\}_{0\leq s\leq t}\xlongequal{d}\{X^{z,t,x}_{s}\}_{0\leq s \leq t}.
\]

As in the Brownian motion case, there is also a stochastic dominance relation between Bessel bridges.  The technical tool we use here is the comparison theorem for solutions of stochastic differential equations and a good reference for it is \cite{ikeda1977comparison}.

\begin{Lemma}\label{stochasticdominance}
If $x_1\geq x_2\geq 0$ and $y_1\geq y_2\geq 0$, then $\{X^{x_1,1,y_1}_r\}_{0\leq r\leq 1}$ stochastically dominates $\{X^{x_2,1,y_2}_r\}_{0\leq r\leq 1}$. In other words, these two processes can be constructed on some probability space such that almost surely for all $r\in [0,1]$,
\[
X^{x_1,1,y_1}_r\geq X^{x_2,1,y_2}_r.
\]
\end{Lemma}
\begin{proof}
It is sufficient to show that for every $0<\delta<1$, the process $\{X^{x_1,1,y_1}_r\}_{0\leq r\leq 1-\delta}$ stochastically dominates $\{X^{x_2,1,y_2}_r\}_{0\leq r\leq 1-\delta}$. Note that the Bessel bridge is nonnegative. Instead of working with Bessel bridges directly, we will prove the lemma for squared Bessel bridges, for which the comparison theorem can be applied readily. 
\noindent
Define squared Bessel bridges for $0\leq r\leq 1-\delta$,
\[
Y^{x_i^2,1,y_i^2}_r:=(X^{x_i,1,y_i}_r)^2, \quad \text{for}\; i=1,2. 
\]
By (0.27) of \cite{pitman2006combinatorial} and Ito's formula, letting $\{B_r\}_{r\geq 0}$ be a standard Brownian motion, squared Bessel bridges $\{Y_r^{x_1^2,1,y_1^2}\}_{0\leq r\leq 1-\delta}$ and $\{Y_r^{x_2^2,1,y_2^2}\}_{0\leq r\leq 1-\delta}$ can be respectively represented as pathwise unique solutions over $[0,1-\delta]$ of the stochastic differential equations
\[
Y^{x_i^2,1,y_i^2}_0=x_i^2, \quad dY^{x_i^2,1,y_i^2}_r=\bigg(3+\frac{2y_i\sqrt{Y^{x_i^2,1,y_i^2}_r}-2Y^{x_i^2,1,y_i^2}_r}{1-r}\bigg)dr+2\sqrt{Y^{x_i^2,1,y_i^2}_{r}}dB_r,\quad\text{for} \; i=1,2.
\]
Set
\[
b_i(t,x)=3+\frac{2y_i\sqrt{x}-2x}{1-t} \;\text{for}\; i=1,2, \quad \sigma(t,x)=2\sqrt{x}.
\]
We see that for $x,y\in \mathbb{R}$ and $t\geq 0$,
\[
|\sigma(t,x)-\sigma(t,y)|=2|\sqrt{x}-\sqrt{y}|\leq 2\sqrt{|x-y|}=:\phi(|x-y|)
\]
where $\phi$ is an increasing function such that $\phi(0)=0$ and 
\[
\int_{0^+}\phi(x)^{-2}dx=\infty.
\]
Furthermore, because $b_i(t,x)$ for $i=1,2$ and $\sigma(t,x)$ are continuous on $[0,1-\delta)\times \mathbb{R}$,  we have $\{Y^{x_1^2,1,y_1^2}_r\}_{0\leq r\leq 1-\delta}$ stochastically dominates $\{Y^{x_2^2,1,y_2^2}_r\}_{0\leq r\leq 1-\delta}$ by Theorem 1.1 of \cite{ikeda1977comparison}. Finally after taking the square root, the lemma holds for Bessel bridges.
\end{proof}

There is also one more fact on the relationship between the Bessel bridge and Bessel process, which is borrowed from Lemma 7 of \cite{harris2007survival}.
\begin{Lemma}\label{weakconvergence}
As $t\rightarrow \infty$, the Bessel bridge converges to the Bessel process in the Skorokhod topology on $D[0,\infty)$, i.e.
\[
P^{z,t,x}_{BES} \Rightarrow P^{z}_{BES}.
\]
\end{Lemma}
 
\section{Upper bound}
\subsection{Proof outline}
In this section, we show the upper bound (\ref{limsup}). Throughout this section, $P^x_{-\rho}$ is the probability measure for the branching Brownian motion started from a single particle at $x>0$ with drift $-\rho$ and absorbed at $0$. Let $\mathcal{N}_t^{-\rho}$ be the set of surviving particles at time $t$. The configuration of particles at time $t$ under $P^x_{-\rho}$ is written as $\{Y_u(t): u\in \mathcal{N}_t^{-\rho}\}$.  For a particle $u\in \mathcal{N}_t^{-\rho}$, denote by $O_u$ the time that the ancestor of $u$ branches off the spine. By convention, if $u$ is the spinal particle, $O_u=t$.
Defined in (\ref{Q}), $Q^x$ is the law of a branching diffusion with the spine which initiates from a single particle at $x>0$. Under the measure $Q^x$, let $\{\xi_t\}_{t\geq 0}$ be the trajectory of the spinal particle which diffuses as a Bessel-3 process. Define $\{\zeta_s\}_{0\leq s \leq t}=\{\xi_{t-s}\}_{0\leq s\leq t}$ to be the reversed trajectory of the spinal particle. We denote by $Q^{x,t,z}$ the law of the branching process whose spine starts from $x$ and is conditioned to end up at $z$ at time $t$, i.e. 
\[
Q^{x,t,z}(\cdot)=Q^x(\cdot|\xi_t=z)=Q^x(\cdot|\zeta_0=z).
\]

First we will control the case where the position of the spinal particle at time $t$ is greater than $\varepsilon^{-1/2}$, which is Lemma \ref{endpointofspine} below. 
\begin{Lemma}\label{endpointofspine}
For all $t$ and all $\varepsilon$ sufficiently small, there exists a positive constant $C_3$ such that
\[
\sqrt{2\pi t^3}Q^x\bigg[\frac{1}{\sum_{u}Y_u(t)e^{\rho Y_u(t)}}; \xi_{t}\geq \varepsilon^{-1/2}\bigg]\leq e^{-C_3/\sqrt{\varepsilon}}.
\] 
\end{Lemma}
\noindent As a result, we only need to deal with the case where the spine ends up near the origin. To prove (\ref{limsup}), it is sufficient to show that there exists a constant $C_4$ such that for $\varepsilon$ small enough,
\begin{equation}
\limsup_{t\rightarrow \infty}\sqrt{2\pi t^3}Q^{x}\bigg[\frac{1}{\sum_{u}Y_u(t)e^{\rho Y_u(t)}};\xi_t\leq \varepsilon^{-1/2}\bigg]\leq e^{-C_{4}/\sqrt{\varepsilon}}.
\label{subgoal}
\end{equation}
It remains to prove equation (\ref{subgoal}). To obtain an upper bound, we only take particles that branch off the spine within the last $\varepsilon^{-3/2}$ time into account. Conditioned on the spine being at $y$ at time $t-\varepsilon^{-3/2}$, we restart the process from $y$ and let the process run for time $\varepsilon^{-3/2}$. Essentially, we will work on bounding
\begin{equation}\label{ubsmallz}
Q^{y}\bigg[\frac{1}{\sum_{u}Y_u(\varepsilon^{-3/2})e^{\rho Y_u(\varepsilon^{-3/2})}}\Big| \xi_{\varepsilon^{-3/2}}=z\bigg].
\end{equation}
Considering the reversed trajectory of the spine, let
\[M=\sup_{0\leq s\leq \varepsilon^{-3/2}}\Big(\varepsilon \zeta_s-\frac{1}{\rho}\varepsilon^2 s\Big).\]
For any positive constant $C>2\pi$, we will divide the proof for (\ref{subgoal}) into the small $M$ and large $M$ cases, 
\begin{align}\label{ubseparatetwoparts}
(\ref{ubsmallz})
&=Q^{y}\bigg[\frac{1}{\sum_{u}Y_u(\varepsilon^{-3/2})e^{\rho Y_u(\varepsilon^{-3/2})}}1_{\{M\geq2C\sqrt{\varepsilon}\}}\Big| \xi_{\varepsilon^{-3/2}}=z\bigg]\nonumber\\
&\quad +Q^{y}\bigg[\frac{1}{\sum_{u}Y_u(\varepsilon^{-3/2})e^{\rho Y_u(\varepsilon^{-3/2})}}1_{\{M < 2C\sqrt{\varepsilon}\}}\Big| \xi_{\varepsilon^{-3/2}}=z\bigg].
\end{align}
For the large $M$ case, the main strategy is as follows:
\begin{itemize}
\item If $M\geq 2C\sqrt{\varepsilon}$, then $\varepsilon \zeta_s-\frac{1}{\rho}\varepsilon^2 s$ stays above $C\sqrt{\varepsilon}$ for a while. In other words, the position of the spine at time $t-s$ satisfies $\xi_{t-s}\geq C/\sqrt{\varepsilon}+\varepsilon s/\rho$ for some time. During that time, many particles branch off the spine.
\item For $s\in [0,\varepsilon^{-3/2}]$, each particle that branches off the spine at time $t-s$ and is located to the right of $C/\sqrt{\varepsilon}+\varepsilon s/\rho$ will have a descendant at time $t$ above $C/(4\sqrt{\varepsilon})$ with some nonzero probability which is independent of $\varepsilon$ and $s$. When this occurs, it will follow that, for sufficiently small $\varepsilon$,
\[
\frac{1}{\sum_{u}Y_u(t)e^{\rho Y_u(t)}}\leq \frac{4\sqrt{\varepsilon}}{C} e^{-C\rho/(4\sqrt{\varepsilon})}\leq e^{-C\rho/(5\sqrt{\varepsilon})}.
\]
\item Taking the number of branching events of the spine into consideration, the probability that there exists at least one particle which stays to the right of $C/(4\sqrt{\varepsilon})$ at time $t$ converges to 1 as $\varepsilon$ goes to 0.
\end{itemize}
From the above strategy, we can see why the $\varepsilon^{-3/2}$ time period is considered here. The length of this time period has to be large enough such that a considerable number of particles branch off the spine and also small enough such that if a particle branches off the spine during that time, the position of its descendant at time $t$ won't be too far from its branching position. Essentially, for $C>2\pi$, we need to prove the following two lemmas.

\begin{Lemma}\label{largeMubPoisson}
Let $\{t-t_i\}_{i=1}^{N_\varepsilon}$ be the set of times that particles branch off the spine between time $t-\varepsilon^{-3/2}$ and $t$. Then there exists a positive constant $C_{5}$ such that  for $\varepsilon$ sufficiently small, for every $y\in(0,\infty)$ and $z \in (0,\varepsilon^{-1/2}]$, we have
\begin{equation*}
Q^{y}\Bigg(\bigg\{\sum_{i=1}^{N_\varepsilon}1_{\{\varepsilon \zeta_{t_i}-\frac{1}{\rho}\varepsilon^{2}t_i\geq C\sqrt{\varepsilon}\}}\leq \frac{1}{\sqrt{\varepsilon}}\bigg\}\cap \{M\geq 2C\sqrt{\varepsilon}\}\bigg|\xi_{\varepsilon^{-3/2}}=z\Bigg)\leq \Big(6+\frac{2}{yz}\Big)e^{-C_5/\sqrt{\varepsilon}}.
\end{equation*}
\end{Lemma}

\begin{Lemma}\label{largeMub}
There exists a positive constant $C_6$ such that for all sufficiently small $\varepsilon$ and $s\in [0,\varepsilon^{-3/2}]$,
\begin{equation}
P^{C/\sqrt{\varepsilon}+\varepsilon s/\rho}_{-\rho}\bigg(\exists u\in \mathcal{N}_s^{-\rho}: Y_u(s)>\frac{C}{4\sqrt{\varepsilon}}\bigg)>C_6.
\label{Lem2}
\end{equation}
\end{Lemma}

\noindent With the help of Lemmas \ref{largeMubPoisson} and \ref{largeMub}, we will be able to state the result regarding the large $M$ case.
\begin{Lemma}\label{largeMubfinal}
There exists a positive constant $C_7$ such that for all sufficiently small $\varepsilon$, for all $y\in (0,\infty)$ and $z\in (0,\varepsilon^{-1/2}]$,
\begin{equation*}
Q^{y}\bigg[\frac{1}{\sum_{u}Y_u(\varepsilon^{-3/2})e^{\rho Y_u(\varepsilon^{-3/2})}}1_{\{M \geq 2C\sqrt{\varepsilon}\}}\bigg| \xi_{\varepsilon^{-3/2}}=z\bigg]\leq \bigg(\frac{2}{yz^2}+\frac{7}{z}+1\bigg)e^{-C_7/\sqrt{\varepsilon}}.
\end{equation*}

\end{Lemma}

\noindent 
As for the small $M$ case, Lemma \ref{smallMub} provides an upper bound. The key step is to bound the probability that $M$ is less than $2C\sqrt{\varepsilon}$.
\begin{Lemma}\label{smallMub}
There exists a positive constant $C_8$ such that for sufficiently small $\varepsilon$, for all $y\in (0,\infty)$ and $z\in (0,\varepsilon^{-1/2}]$,
\begin{equation}\label{smallMubeq}
Q^{y}\bigg[\frac{1}{\sum_{u}Y_u(\varepsilon^{-3/2})e^{\rho Y_u(\varepsilon^{-3/2})}}1_{\{M < 2C\sqrt{\varepsilon}\}}\bigg| \xi_{\varepsilon^{-3/2}}=z\bigg]\leq \frac{1}{z}e^{-C_8/\sqrt{\varepsilon}}.
\end{equation}
\end{Lemma}
\noindent
In the end, the upper bound (\ref{limsup}) is proved in Section 3.2 by combining Lemmas \ref{endpointofspine}, \ref{largeMubfinal} and \ref{smallMub}.

In Section 3.2, we will gather all the lemmas to obtain the upper bound  (\ref{limsup}) and in Section 3.3, we will provide proofs for the lemmas above. 

\subsection{Proof of upper bound}

To begin with, conditioning on the end point of the spinal trajectory, we have
\begin{align}\label{Lem1Bessel}
&\sqrt{2\pi t^3}Q^{x}\bigg[\frac{1}{\sum_{u}Y_u(t)e^{\rho Y_u(t)}};\xi_t\leq \varepsilon^{-1/2}\bigg] \nonumber\\
&\hspace{0.2in}= \sqrt{2\pi t^3}\int_{0}^{\varepsilon^{-1/2}}Q^{x}\bigg[\frac{1}{\sum_{u}Y_u(t)e^{\rho Y_u(t)}}\bigg|\xi_t=z\bigg]\frac{z}{x}\frac{1}{\sqrt{2\pi t}}e^{-(x-z)^2/2t}(1-e^{-2xz/t})dz\nonumber\\
&\hspace{0.2in}\leq \sqrt{2\pi t^3}\int_{0}^{\varepsilon^{-1/2}}Q^{x}\bigg[\frac{1}{\sum_{u}Y_u(t)e^{\rho Y_u(t)}}\bigg|\xi_t=z\bigg]\frac{z}{x}\frac{1}{\sqrt{2\pi t}}e^{-(x-z)^2/2t}\frac{2xz}{t}dz\nonumber\\
&\hspace{0.2in}=\int_{0}^{\varepsilon^{-1/2}}2z^2e^{-(x-z)^2/2t}Q^{x}\bigg[\frac{1}{\sum_{u}Y_u(t)e^{\rho Y_u(t)}}\bigg|\xi_t=z\bigg]dz.
\end{align}
Next, knowing that a reversed Bessel bridge is still a Bessel bridge, if $\{\xi_s\}_{0\leq s\leq t}$ is a Bessel bridge from $x$ to $z$ within time $t$, then $\{\zeta_s\}_{0\leq s\leq t}$ is also a Bessel bridge from $z$ to $x$ within time $t$. Since we are going to obtain an upper bound, it is enough to only look at the set of living particles at time $t$ that branch off the spine in the last $\varepsilon^{-3/2}$ time. For clarification, under $Q^x$, the set $\{u\in \mathcal{N}_{t}:O_u\geq t-\varepsilon^{-3/2}\}$ includes the spinal particle. We have
\begin{align}\label{Lem1Separate}
Q^{x}\bigg[\frac{1}{\sum_{u}Y_u(t)e^{\rho Y_u(t)}}\bigg|\xi_t=z\bigg]
& \leq Q^{x,t,z}\bigg[\frac{1}{\sum_{u}Y_u(t)e^{\rho Y_u(t)}1_{\{O_u \geq t-\varepsilon^{-3/2}\}}}\bigg] \nonumber\\
&=\int_{0}^{\infty}Q^{x,t,z}\bigg[\frac{1}{\sum_{u}Y_u(t)e^{\rho Y_u(t)}1_{\{O_u \geq t-\varepsilon^{-3/2}\}}}\bigg| \zeta_{\varepsilon^{-3/2}}=y\bigg]p_{\varepsilon^{-3/2}}^{z,t,x}(y)dy.
\end{align}
According to the Markov property of branching Brownian motion and the Bessel process,
\begin{align}
&Q^{x,t,z}\bigg[\frac{1}{\sum_{u}Y_u(t)e^{\rho Y_u(t)}1_{\{O_u \geq t-\varepsilon^{-3/2}\}}}\bigg| \zeta_{\varepsilon^{-3/2}}=y\bigg]=Q^{y,\varepsilon^{-3/2},z}\bigg[\frac{1}{\sum_{u}Y_u(\varepsilon^{-3/2})e^{\rho Y_u(\varepsilon^{-3/2})}}\bigg].
\label{Lem1I1}
\end{align}
Note that $1-e^{-x}\leq x$ for all $x\geq 0$ and $1-e^{-x}\geq x/2$ for $0\leq x\leq 1$. For any fixed $\varepsilon$, if $t$ is large enough such that $t/(t-\varepsilon^{-3/2})\leq 2$, then for all $z\in (0,\varepsilon^{-1/2}]$ and $y\in (0,\infty)$, 
\begin{align}\label{density}
p_{\varepsilon^{-3/2}}^{z,t,x}(y)
&=\frac{p_{\varepsilon^{-3/2}}(z,y)p_{t-\varepsilon^{-3/2}}(y,x)}{p_{t}(z,x)}\nonumber\\
&=\frac{1}{\sqrt{2\pi\varepsilon^{-3/2}}}\cdot\frac{y}{z}e^{-(y-z)^2/(2\varepsilon^{-3/2})}(1-e^{-2yz\varepsilon^{3/2}})\nonumber\\
&\quad\times\frac{\frac{1}{\sqrt{2\pi(t-\varepsilon^{-3/2})}}\cdot \frac{x}{y}e^{-(y-x)^2/2(t-\varepsilon^{-3/2})}(1-e^{-2yx/(t-\varepsilon^{-3/2})})}{\frac{1}{\sqrt{2\pi t}}\cdot \frac{x}{z}e^{-(z-x)^2/2t}(1-e^{-2xz/t})}\nonumber\\
&\leq \frac{1}{\sqrt{2\pi\varepsilon^{-3/2}}}\frac{y}{z}e^{-(y-z)^2/(2\varepsilon^{-3/2})}2yz\varepsilon^{3/2}\sqrt{\frac{t}{t-\varepsilon^{-3/2}}}\frac{z}{y}e^{(z-x)^2/2t}\frac{2xy}{t-\varepsilon^{-3/2}}\frac{t}{xz}\nonumber\\
&=\sqrt{\frac{8}{\pi}}\varepsilon^{9/4}\bigg(\frac{t}{t-\varepsilon^{-3/2}}\bigg)^{3/2}y^2e^{-(y-z)^2/(2\varepsilon^{-3/2})}e^{(z-x)^2/2t}\nonumber\\
&\leq \frac{8}{\sqrt{\pi}}\varepsilon^{9/4}y^2e^{-(y-z)^2/(2\varepsilon^{-3/2})}e^{(z-x)^2/2t}.
\end{align}
By (\ref{Lem1Bessel}), (\ref{Lem1Separate}), (\ref{Lem1I1}) and (\ref{density}), for every fixed $\varepsilon$, if $t$ is large enough, we have
\begin{align*}
&\sqrt{2\pi t^3}Q^{x}\bigg[\frac{1}{\sum_{u}Y_u(t)e^{\rho Y_u(t)}};\xi_t\leq \varepsilon^{-1/2}\bigg]\\
&\hspace{0.2in}\leq \int_{0}^{\varepsilon^{-1/2}}\int_{0}^{\infty}\frac{16}{\sqrt{\pi}}\varepsilon^{9/4}Q^{y,\varepsilon^{-3/2},z}\bigg[\frac{1}{\sum_{u}Y_u(\varepsilon^{-3/2})e^{\rho Y_u(\varepsilon^{-3/2})}}\bigg]z^2y^2e^{-(y-z)^2/(2\varepsilon^{-3/2})}dydz.
\end{align*}
Furthermore, by Lemmas \ref{largeMubfinal} and \ref{smallMub}, it follows that for sufficiently small $\varepsilon$, if $t$ is large enough,
\begin{align}\label{ub}
&\sqrt{2\pi t^3}Q^{x}\bigg[\frac{1}{\sum_{u}Y_u(t)e^{\rho Y_u(t)}};\xi_t\leq \varepsilon^{-1/2}\bigg]\nonumber\\
&\hspace{0.2in}\leq \int_{0}^{\varepsilon^{-1/2}}\int_{0}^{\infty}\frac{16}{\sqrt{\pi}}\varepsilon^{9/4}\bigg[\bigg(\frac{2}{yz^2}+\frac{7}{z}+1\bigg)e^{-C_7/\sqrt{\varepsilon}}+\frac{1}{z}e^{-C_8/\sqrt{\varepsilon}}\bigg]z^2y^2e^{-(y-z)^2/(2\varepsilon^{-3/2})}dydz\nonumber\\
&\hspace{0.2in}\leq\frac{16}{\sqrt{\pi}}\varepsilon^{9/4}\bigg\{ \int_{0}^{\varepsilon^{-1/2}}\int_{0}^{\infty}2e^{-C_7/\sqrt{\varepsilon}}ye^{-(y-z)^2/(2\varepsilon^{-3/2})}dydz\nonumber\\
&\hspace{0.2in}\quad+ \int_{0}^{\varepsilon^{-1/2}}\int_{0}^{\infty}\Big(7e^{-C_7/\sqrt{\varepsilon}}+e^{-C_8/\sqrt{\varepsilon}}\Big)zy^2e^{-(y-z)^2/(2\varepsilon^{-3/2})}dydz\nonumber\\
&\hspace{0.2in}\quad+\int_{0}^{\varepsilon^{-1/2}}\int_{0}^{\infty}e^{-C_7/\sqrt{\varepsilon}}z^2y^2e^{-(y-z)^2/(2\varepsilon^{-3/2})}dydz\bigg\}\nonumber\\
&\hspace{0.2in}=:\frac{16}{\sqrt{\pi}}\varepsilon^{9/4}(I_1+I_2+I_3).
\end{align}
For the first term, by substitution, we obtain that
\begin{align}\label{I_1}
I_1&=2e^{-C_7/\sqrt{\varepsilon}}\int_{0}^{\varepsilon^{-1/2}}\int_{-z}^{\infty}(u+z)e^{-u^2/(2\varepsilon^{-3/2})}dudz\nonumber\\
&\leq 2e^{-C_7/\sqrt{\varepsilon}}\bigg(\int_{0}^{\varepsilon^{-1/2}}\int_{0}^{\infty}ue^{-u^2/(2\varepsilon^{-3/2})}dudz+\int_{0}^{\varepsilon^{-1/2}}z\int_{-\infty}^{\infty}e^{-u^2/(2\varepsilon^{-3/2})}dudz\bigg)\nonumber\\
&= \bigg(2\varepsilon^{-2}+\sqrt{2\pi}\varepsilon^{-7/4}\bigg)e^{-C_7/\sqrt{\varepsilon}}.
\end{align}
As for $I_2$ and $I_3$, because
\begin{align*}
\int_{0}^{\infty}y^2e^{-(y-z)^2/(2\varepsilon^{-3/2})}dy&\leq\sqrt{2\pi\varepsilon^{-3/2}}\int_{-\infty}^{\infty}y^2\frac{1}{\sqrt{2\pi\varepsilon^{-3/2}}}e^{-(y-z)^2/(2\varepsilon^{-3/2})}dy\\
&=\sqrt{2\pi\varepsilon^{-3/2}}\Big(\varepsilon^{-3/2}+z^2\Big),
\end{align*}
we have
\begin{align}\label{I_2}
I_2
&\leq \sqrt{2\pi}\varepsilon^{-3/4}\Big(7e^{-C_7/\sqrt{\varepsilon}}+e^{-C_8/\sqrt{\varepsilon}}\Big)\int_{0}^{\varepsilon^{-1/2}}z(\varepsilon^{-3/2}+z^2)dz\nonumber\\
&\leq \sqrt{2\pi}\varepsilon^{-13/4}\Big(7e^{-C_7/\sqrt{\varepsilon}}+e^{-C_8/\sqrt{\varepsilon}}\Big),
\end{align}
and
\begin{align}\label{I_3}
I_3&\leq \sqrt{2\pi}\varepsilon^{-3/4}e^{-C_7/\sqrt{\varepsilon}}\int_{0}^{\varepsilon^{-1/2}} z^2(\varepsilon^{-3/2}+z^2)dz\nonumber\\
&\leq \sqrt{2\pi}\varepsilon^{-15/4}e^{-C_7/\sqrt{\varepsilon}}.
\end{align}
Setting $0<C_4<\min\{C_7,C_8\}$, equation (\ref{subgoal}) follows from (\ref{ub}), (\ref{I_1}), (\ref{I_2}) and (\ref{I_3}). Finally, according to Lemma \ref{endpointofspine} and equation (\ref{subgoal}), the upper bound (\ref{limsup}) is proved by letting $0< C_1<\min\{C_3,C_4\}$. 

\subsection{Proofs of Lemmas}
\noindent\textit{Proof of Lemma \ref{endpointofspine}}.
For all $\varepsilon$ and all $t$, we have a trivial bound for the expectation
\begin{align*}
&\sqrt{2\pi t^3}Q^x\bigg[\frac{1}{\sum_{u}Y_u(t)e^{\rho Y_u(t)}}; \xi_{t}\geq \varepsilon^{-1/2}\bigg]\\
& \hspace{0.2in}\leq \sqrt{2\pi t^3}Q^x\bigg[\frac{1}{\xi_{t}e^{\rho \xi_{t}}};\xi_t\geq \varepsilon^{-1/2}\bigg]\\
&\hspace{0.2in}=\sqrt{2\pi t^3}\int_{\varepsilon^{-1/2}}^{\infty}\frac{1}{ze^{\rho z}}p_t(x,z)dz\\
&\hspace{0.2in} =\sqrt{2\pi t^3} \int_{\varepsilon^{-1/2}}^{\infty}\frac{1}{ze^{\rho z}}\frac{z}{x}\frac{1}{\sqrt{2\pi t}}e^{-(x-z)^2/2t}(1-e^{-2xz/t})dz\\
&\hspace{0.2in}\leq \sqrt{2\pi t^3}\int_{\varepsilon^{-1/2}}^{\infty}\frac{1}{ze^{\rho z}}\frac{z}{x}\frac{1}{\sqrt{2\pi t}}\frac{2xz}{t}dz\\
&\hspace{0.2in}=2\int_{\varepsilon^{-1/2}}^{\infty}ze^{-\rho z}dz\\
&\hspace{0.2in}=\frac{2}{\rho}e^{-\rho /\sqrt{\varepsilon}}\bigg(\frac{1}{\sqrt{\varepsilon}}+\frac{1}{\rho}\bigg).
\end{align*}
Letting $0<C_3<\sqrt{2}\leq\rho$, Lemma \ref{endpointofspine} is established.
\qedwhite
\\

For the rest of this section we will denote by $X^{a,t,b}$ a Bessel bridge from $a$ to $b$ in time $t$ and $B^{a,t,b}$ a Brownian bridge from $a$ to $b$ in time $t$. 
\\
\noindent\textit{Proof of Lemma \ref{largeMubPoisson}}.
Observe that since $C+1/\rho\leq 3C/2$,  for $0\leq t_i\leq \varepsilon^{-3/2}$, 
\begin{align*}
&Q^{y}\Bigg(\bigg\{\sum_{i=1}^{N_\varepsilon}1_{\{\varepsilon \zeta_{t_i}-\frac{1}{\rho}\varepsilon^{2}t_i\geq C\sqrt{\varepsilon}\}}\leq \frac{1}{\sqrt{\varepsilon}}\bigg\}\cap \big\{M\geq 2C\sqrt{\varepsilon}\big\}\bigg| \xi_{\varepsilon^{-3/2}}=z\Bigg)\\
&\hspace{0.2in}\leq Q^{y,\varepsilon^{-3/2},z}\Bigg(\bigg\{\sum_{i=1}^{N_\varepsilon}1_{\{\zeta_{t_i}\geq 3C/(2\sqrt{\varepsilon})\}}\leq \frac{1}{\sqrt{\varepsilon}}\bigg\}\cap \big\{M\geq 2C\sqrt{\varepsilon}\big\}\Bigg).
\end{align*}
So it is sufficient to show that there exists a constant $C_{5}$ such that for every $y\in(0,\infty)$ and $z\in (0,\varepsilon^{-1/2}]$,
\begin{equation}\label{largeMubPoissonfirstsimplification}
Q^{y,\varepsilon^{-3/2},z}\Bigg(\bigg\{\sum_{i=1}^{N_\varepsilon}1_{\{\zeta_{t_i}\geq 3C/(2\sqrt{\varepsilon})\}}\leq \frac{1}{\sqrt{\varepsilon}}\bigg\}\cap \big\{M\geq 2C\sqrt{\varepsilon}\big\}\Bigg) \leq \bigg(4+\frac{1}{yz}\bigg)e^{-C_{5}/\sqrt{\varepsilon}}.
\end{equation}
According to the spinal decomposition and the formula for expectations of additive functionals of Poisson point processes,
\begin{equation*}
Q^{y,\varepsilon^{-3/2},z}\bigg(\sum_{i=1}^{N_\varepsilon}1_{\{\zeta_{t_i}\geq 3C/(2\sqrt{\varepsilon})\}}\bigg|\{\zeta_s\}_{0\leq s\leq \varepsilon^{-3/2}}\bigg)=2\int_{0}^{\varepsilon^{-3/2}}1_{\{\zeta_s\geq 3C/(2\sqrt{\varepsilon})\}}ds.
\end{equation*}
Below for simplicity, we denote
\[ 
\sum_{i=1}^{N_\varepsilon}1_{\{\zeta_{t_i}\geq 3C/(2\sqrt{\varepsilon})\}}=:X,
\]
\[
2\int_{0}^{\varepsilon^{-3/2}}1_{\{\zeta_s\geq 3C/(2\sqrt{\varepsilon})\}}ds=:Y.
\]
The proof of (\ref{largeMubPoissonfirstsimplification}) can be separated into two parts with the help of the above conditional expectation,
\begin{align}\label{largeMubPoissonI}
&Q^{y,\varepsilon^{-3/2},z}\bigg( \Big\{X\leq \frac{1}{\sqrt{\varepsilon}}\Big\}\cap \{M\geq 2C\sqrt{\varepsilon}\}\bigg) \nonumber\\
&\hspace{0.2in}=Q^{y,\varepsilon^{-3/2},z}\bigg(\Big\{X\leq\frac{1}{\sqrt{\varepsilon}}\Big\}\cap \Big\{Y\geq \frac{2}{\sqrt{\varepsilon}}\Big\}\cap \{M\geq 2C\sqrt{\varepsilon}\}\bigg) \nonumber\\
&\hspace{0.2in}\quad + Q^{y,\varepsilon^{-3/2},z}\bigg(\Big\{X\leq \frac{1}{\sqrt{\varepsilon}}\Big\}\cap \Big\{Y\leq \frac{2}{\sqrt{\varepsilon}}\Big\}\cap \{M\geq 2C\sqrt{\varepsilon}\}\bigg)\nonumber\\
&\hspace{0.2in}\leq Q^{y,\varepsilon^{-3/2},z}\bigg(\Big\{X\leq\frac{1}{\sqrt{\varepsilon}}\Big\}\cap \Big\{Y\geq \frac{2}{\sqrt{\varepsilon}}\Big\}\bigg) +Q^{y,\varepsilon^{-3/2},z}\bigg( \Big\{Y\leq \frac{2}{\sqrt{\varepsilon}}\Big\}\cap \{M\geq 2C\sqrt{\varepsilon}\}\bigg)\nonumber\\
&\hspace{0.2in}=: J_1+J_2.
\end{align}
First, we will show that conditioned on the trajectory of the spine, the number of particles that branch off the spine at a position and time $(\zeta_{t_i}, t_i)$ satisfying $\zeta_{t_i} \geq 3C/(2\sqrt{\varepsilon})$ isn't far from its conditional expectation, which gives an upper bound for $J_1$. Next we will find an upper bound for $J_2$ through analysis of the behavior of the spine. 
 
For the first part, we will apply the following bound for the Poisson distribution (see, e.g. \cite{chernoffcs}). For a Poisson distributed random variable $Z$ with expectation $\lambda$, for any $v>0$, we have
\begin{equation} \label{Chernoff}
P(|Z-\lambda|\geq v)\leq 2e^{-v^2/2(\lambda+v)}.
\end{equation}
\noindent 
We know that under $Q^{y,\varepsilon^{-3/2},z}$, the conditional distribution of $X$ given $\{\zeta_s\}_{0\leq s\leq \varepsilon^{-3/2}}$ is a Poisson distribution with parameter $Y$. Applying (\ref{Chernoff}) under the conditional expectation with $\lambda=Y$ and $v=Y/2$, we have 
\begin{align}\label{largeMubPoissonI1}
J_1&\leq Q^{y,\varepsilon^{-3/2},z}\Big(1_{\{Y\geq 2/\sqrt{\varepsilon}\}} Q^{y,\varepsilon^{-3/2},z}\big(\big|X-Y\big|\geq Y/2\big|\{\zeta_s\}_{0\leq s\leq \varepsilon^{-3/2}} \big)\Big) \nonumber\\
& \leq Q^{y,\varepsilon^{-3/2},z}\bigg(1_{\{Y\geq 2/\sqrt{\varepsilon}\}}2\exp{\Big\{-\frac{(Y/2)^2}{2(Y+Y/2)}\Big\}}\bigg)\nonumber\\
&\leq 2Q^{y,\varepsilon^{-3/2},z}\Big(1_{\{Y\geq 2/\sqrt{\varepsilon}\}}e^{-Y/12}\Big)\nonumber\\
&\leq 2e^{-1/(6\sqrt{\varepsilon})}.
\end{align}

\noindent As for the second part, we have
\begin{align*}\label{largeMubPoissonI2scaling}
J_2
&\leq Q^{y,\varepsilon^{-3/2},z}\Bigg(\bigg\{\int_{0}^{\varepsilon^{-3/2}}1_{\{\zeta_s\geq 3C/(2\sqrt{\varepsilon})\}}ds\leq \frac{1}{\sqrt{\varepsilon}}\bigg\}\cap \bigg\{\sup_{0\leq s\leq \varepsilon^{-3/2}}\varepsilon\zeta_s \geq 2C\sqrt{\varepsilon}\bigg\}\Bigg)\\
&=Q^{y,\varepsilon^{-3/2},z}\Bigg(\bigg\{\int_{0}^{1}1_{\{\varepsilon^{3/4}\zeta_{\varepsilon^{-3/2}r}\geq 3C\varepsilon^{1/4}/2\}}dr\leq \varepsilon\bigg\}\cap\bigg\{\sup_{0\leq r\leq 1}\varepsilon^{3/4}\zeta_{\varepsilon^{-3/2}r}\geq 2C\varepsilon^{1/4}\bigg\}\bigg).
\end{align*}
Notice that under $Q^{y,\varepsilon^{-3/2},z}$, the process $\{\zeta_s\}_{0\leq s\leq \varepsilon^{-3/2}}$ is a Bessel bridge from $z$ to $y$ in time $\varepsilon^{-3/2}$. After scaling, $\{\varepsilon^{3/4}\zeta_{\varepsilon^{-3/2}r}\}_{0\leq r\leq 1}$  is a Bessel bridge from $\varepsilon^{3/4}z$ to $\varepsilon^{3/4}y$ within time $1$. Recall that $\{X_r^{\varepsilon^{3/4}z,1,\varepsilon^{3/4}y}\}_{0\leq r\leq 1}$ is a Bessel bridge from $\varepsilon^{3/4}z$ to $\varepsilon^{3/4}y$ within time $1$. For simplicity, we will write $\{X_r\}_{0\leq r\leq 1}$ in place of $\{X_r^{\varepsilon^{3/4}z,1,\varepsilon^{3/4}y}\}_{0\leq r\leq 1}$. Therefore, we have 
\begin{equation}\label{largeMubPoissonI2scaling}
J_2\leq P\Bigg(\bigg\{\int_{0}^{1}1_{\{X_r\geq 3C\varepsilon^{1/4}/2\}}dr\leq \varepsilon\bigg\}\cap\bigg\{\sup_{0\leq r\leq 1}X_r\geq 2C\varepsilon^{1/4}\bigg\}\Bigg).
\end{equation}
Define $\{\bar{X}_r\}_{0\leq r\leq 1}=\{X_{1-r}\}_{0\leq r\leq 1}$ to be the time reversed process of $\{X_r\}_{0\leq r\leq 1}$. Then $\{\bar{X}_r\}_{0\leq r\leq 1}$ is a Bessel bridge from $\varepsilon^{3/4}y$ to $\varepsilon^{3/4}z$ in time $1$. Thus the intersection of the events in (\ref{largeMubPoissonI2scaling}) is contained in the union of two events. One of the events is that $\{X_r\}_{0\leq r\leq 1}$ first reaches $2C\varepsilon^{1/4}$ before time $1/2$ and then comes down below $3C\varepsilon^{1/4}/2$ in time less than $\varepsilon$. The other event is that $\{\bar{X}_r\}_{0\leq r\leq 1}$ first reaches $2C\varepsilon^{1/4}$ before time $1/2$ and then comes down below $3C\varepsilon^{1/4}/2$ in time less than $\varepsilon$. Define 
\[
\tau=\inf \big\{r\geq 0: X_r\geq 2C\varepsilon^{1/4}\big\}, \quad \bar{\tau}=\inf \big\{r\geq 0: \bar{X}_r\geq 2C\varepsilon^{1/4}\big\}.
\]
We see that
\begin{align}\label{largeMubPoissontwoevents}
&P\Bigg(\bigg\{\int_{0}^{1}1_{\{X_r\geq 3C\varepsilon^{1/4}/2\}}dr\leq \varepsilon\bigg\}\cap\bigg\{\sup_{0\leq r\leq 1}X_r\geq 2C\varepsilon^{1/4}\bigg\}\Bigg)\nonumber\\
& \hspace{0.2in}\leq P\bigg(\Big\{\tau\leq\frac{1}{2}\Big\}\cap \bigg\{\min_{0\leq r\leq \varepsilon}X_{\tau+r}\leq \frac{3C\varepsilon^{1/4}}{2}\bigg\}\bigg)+P\bigg(\Big\{\bar{\tau}\leq\frac{1}{2}\Big\}\cap \bigg\{\min_{0\leq r\leq \varepsilon}\bar{X}_{\bar{\tau}+r}\leq \frac{3C\varepsilon^{1/4}}{2}\bigg\}\bigg).
\end{align}
Therefore, the proof for the second part of Lemma \ref{largeMubPoisson} boils down to Lemma \ref{largeMubPoissonfirstevent}, whose statement and proof is deferred until later. 

Letting $0<C_5<\min\{1/6, C_9\}$, with equations (\ref{largeMubPoissonI}),  (\ref{largeMubPoissonI1}),  (\ref{largeMubPoissonI2scaling}),  (\ref{largeMubPoissontwoevents}) and Lemma  \ref{largeMubPoissonfirstevent}, formula (\ref{largeMubPoissonfirstsimplification}) is proved and thus the proof of Lemma  \ref{largeMubPoisson} is finished.

\qedwhite
\\

Below, we will state and prove Lemma \ref{largeMubPoissonfirstevent}.
\begin{Lemma}\label{largeMubPoissonfirstevent}
There exists a positive constant $C_{9}$ such that for $\varepsilon$ sufficiently small, for all $z\in (0,\infty)$ and $y\in (0,\infty)$,
\begin{equation}\label{largeMubPoissonfirsteventformula}
P\bigg(\Big\{\tau\leq\frac{1}{2}\Big\}\cap \bigg\{\min_{0\leq r\leq \varepsilon}X_{\tau+r}\leq \frac{3C\varepsilon^{1/4}}{2}\bigg\}\bigg)\leq \bigg(\frac{1}{yz}+2\bigg)e^{-C_{9}/\sqrt{\varepsilon}}.
\end{equation}
\end{Lemma}
\noindent\textit{Proof of Lemma \ref{largeMubPoissonfirstevent}}.
Let's first consider the case when $z\in (2C\varepsilon^{-1/2},\infty)$. Under this scenario, $\tau=0$ and thus
\begin{equation}
P\bigg(\Big\{\tau<\frac{1}{2}\Big\}\cap \bigg\{\min_{0\leq r\leq \varepsilon}X_{\tau+r}\leq\frac{3C\varepsilon^{1/4}}{2}\bigg\}\bigg)
=P\bigg(\min_{0\leq r\leq \varepsilon}X_{r}\leq \frac{3C\varepsilon^{1/4}}{2}\bigg).
\end{equation}
According to Lemma \ref{stochasticdominance}, the process $\{X_r\}_{0\leq r\leq 1}$ stochastically dominates  $\{X_r^{2C\varepsilon^{1/4}, 1, \varepsilon^{3/4}y}\}_{0\leq r\leq 1}$, which is a Bessel bridge from $2C\varepsilon^{1/4}$ to $\varepsilon^{3/4}y$ in time $1$. Therefore
\begin{align}
P\bigg(\min_{0\leq r\leq \varepsilon}X_{r}\leq \frac{3C\varepsilon^{1/4}}{2}\bigg)
&\leq P\bigg(\min_{0\leq r\leq \varepsilon}X_r^{2C\varepsilon^{1/4}, 1, \varepsilon^{3/4}y}\leq \frac{3C\varepsilon^{1/4}}{2}\bigg)\\
&=P\bigg(\Big\{\tau\leq\frac{1}{2}\Big\}\cap \bigg\{\min_{0\leq r\leq \varepsilon}X^{2C\varepsilon^{1/4},1,\varepsilon{3/4y}}_{\tau+r}\leq \frac{3C\varepsilon^{1/4}}{2}\bigg\}\bigg).
\end{align}
Therefore, it is sufficient to only consider the case when $z\in (0,2C\varepsilon^{-1/2}]$. 

For $y,z>0$, denote by $\{B_r^{\varepsilon^{3/4}z,1,\varepsilon^{3/4}y}\}_{0\leq r\leq 1}$ a Brownian bridge from $\varepsilon^{3/4}z$ to $\varepsilon^{3/4}y$ within time 1. Define 
\[
\tau_0=\tau_0(y,z):=\inf \Big\{r\in[0,1]:B_{r}^{\varepsilon^{3/4}z,1,\varepsilon^{3/4}y}=0\Big\}
\] 
and $\tau'$ to be $\tau$ under the setting of Brownian bridge 
\[
\tau':=\inf \Big\{r\geq 0: B_r^{\varepsilon^{3/4}z,1,\varepsilon^{3/4}y}\geq 2C\varepsilon^{1/4}\Big\}.
\] 
By convention, $\inf\emptyset=\infty$. We know that (see, e.g., page 86 of \cite{harris2007survival}) the probability that a Brownian bridge avoids the origin is 
\[
P(\tau_0=\infty)=1-e^{-2\varepsilon^{3/2}yz}.
\]
Furthermore, according to the first part of the proof of Lemma 7 in \cite{harris2007survival}, a Brownian bridge that is conditioned to avoid the origin has the same law as a Bessel bridge. Together with the inequality
\[
\frac{1}{1-e^{-x}}\leq \frac{2}{x}1_{\{0<x<1\}}+2\cdot1_{\{x\geq 1\}}\leq \frac{2}{x}+2,
\]
we have for $\varepsilon$ sufficiently small,
\begin{align}\label{largeMubPoissonfirsteventBM}
&P\bigg(\Big\{\tau\leq\frac{1}{2}\Big\}\cap \bigg\{\min_{0\leq r\leq \varepsilon}X_{\tau+r}\leq \frac{3C\varepsilon^{1/4}}{2}\bigg\}\bigg)\nonumber\\
&\hspace{0.2in}=P\bigg(\Big\{\tau\leq\frac{1}{2}\Big\}\cap \bigg\{\min_{0\leq r\leq \varepsilon}\Big(X_{\tau+r}-X_{\tau}\Big)\leq -\frac{C\varepsilon^{1/4}}{2}\bigg\}\bigg)\nonumber\\
&\hspace{0.2in}=P\bigg(\Big\{\tau'\leq\frac{1}{2}\Big\}\cap \bigg\{\min_{0\leq r\leq \varepsilon}\Big(B^{\varepsilon^{3/4}z,1,\varepsilon^{3/4}y}_{\tau'+r}-B^{\varepsilon^{3/4}z,1,\varepsilon^{3/4}y}_{\tau'}\Big)\leq -\frac{C\varepsilon^{1/4}}{2}\bigg\}\bigg| \tau_0=\infty\bigg)\nonumber\\
&\hspace{0.2in}=\frac{P\Big(\big\{\tau'\leq\frac{1}{2}\big\}\cap \Big\{\min_{0\leq r\leq \varepsilon}\Big(B^{\varepsilon^{3/4}z,1,\varepsilon^{3/4}y}_{\tau'+r}-B^{\varepsilon^{3/4}z,1,\varepsilon^{3/4}y}_{\tau'}\Big)\leq -C\varepsilon^{1/4}/2\Big\}\cap \{\tau_0=\infty\}\Big)}{P(\tau_0=\infty)}\nonumber\\
&\hspace{0.2in}\leq \bigg(\frac{1}{\varepsilon^{3/2}yz}+2\bigg)P\bigg(\Big\{\tau'\leq\frac{1}{2}\Big\}\cap \bigg\{\min_{0\leq r\leq \varepsilon}(B^{\varepsilon^{3/4}z,1,\varepsilon^{3/4}y}_{\tau'+r}-B^{\varepsilon^{3/4}z,1,\varepsilon^{3/4}y}_{\tau'})\leq -\frac{C\varepsilon^{1/4}}{2}\bigg\}\bigg).
\end{align}
Now we are going to bound the probability of the above event under the setting of the Brownian bridge. Let $\mathcal{F}_{\tau'}$ be the $\sigma-$field generated by the stopping time $\tau'$. Conditioning on $\mathcal{F}_{\tau'}$,
\begin{align}
\label{largeMubPoissonfirsteventcond}
&P\bigg(\Big\{\tau'\leq\frac{1}{2}\Big\}\cap \bigg\{\min_{0\leq r\leq \varepsilon}(B^{\varepsilon^{3/4}z,1,\varepsilon^{3/4}y}_{\tau'+r}-B^{\varepsilon^{3/4}z,1,\varepsilon^{3/4}y}_{\tau'})\leq -\frac{C\varepsilon^{1/4}}{2}\bigg\}\bigg)\nonumber\\
&\hspace{0.2in}=E\bigg[1_{\{\tau'\leq\frac{1}{2}\}}P \bigg(\min_{0\leq r\leq \varepsilon}\Big(B^{\varepsilon^{3/4}z,1,\varepsilon^{3/4}y}_{\tau'+r}-B^{\varepsilon^{3/4}z,1,\varepsilon^{3/4}y}_{\tau'}\Big)\leq -\frac{C\varepsilon^{1/4}}{2}\bigg|\mathcal{F}_{\tau'}\bigg)\bigg].
\end{align}
Since the Brownian bridge is a strong Markov process (see, e.g., Proposition 1 of \cite{fitzsimmons1993markovian}), the conditional distribution of $\min_{0\leq r\leq \varepsilon}\big(B^{\varepsilon^{3/4}z,1,\varepsilon^{3/4}y}_{\tau'+r}-B_{\tau'}^{\varepsilon^{3/4}z,1,\varepsilon^{3/4}y}\big)$ given $\tau'=1-u$ is the same as the distribution of $\min_{0\leq r\leq \varepsilon}B^{0,u,\varepsilon^{3/4}y-2C\varepsilon^{1/4}}_{r}$ and is independent of $\mathcal{F}_{\tau'}$. Therefore, given $\tau'=1-u$, the probability inside equation (\ref{largeMubPoissonfirsteventcond}) can be written as
\begin{equation}\label{largeMubPoissonfirsteventstrong}
P \bigg(\min_{0\leq r\leq \varepsilon}B^{0,u,\varepsilon^{3/4}y-2C\varepsilon^{1/4}}_{r}\leq -\frac{C\varepsilon^{1/4}}{2} \bigg)=P\bigg(\max_{0\leq r\leq \varepsilon}B^{0,u,-\varepsilon^{3/4}y+2C\varepsilon^{1/4}} _{r}\geq \frac{C\varepsilon^{1/4}}{2}\bigg).
\end{equation}
To bound the probability inside the expectation, we will consider the cases where $y\in (0,\varepsilon^{-1/2}]$ and $y\in (\varepsilon^{-1/2}, \infty)$ separately. For $y\in (0,\varepsilon^{-1/2}]$, we will apply Theorem 2.1 of \cite{beghin1999maximum}, which gives the distribution of the maximum of the beginning period of a Brownian bridge. Let $\beta=C\varepsilon^{1/4}/2$, $\eta=2C\varepsilon^{1/4}-\varepsilon^{3/4}y$ and $s=\varepsilon$. We have
\begin{align}\label{largeMubPoissonfirsteventref}
&P \bigg(\max_{0\leq r\leq \varepsilon}B^{0,u,-\varepsilon^{3/4}y+2C\varepsilon^{1/4}}_{r}\geq \frac{C\varepsilon^{1/4}}{2}\bigg)\nonumber\\
&\hspace{0.2in}=\exp{\Big\{-\frac{2\beta(\beta-\eta)}{u}\Big\}}\int_{-\infty}^{(2\beta s-\eta s-\beta u)/\sqrt{us(u-s)}}\frac{e^{-v^2/2}}{\sqrt{2\pi}}dv+\int_{(\beta u-\eta s)/\sqrt{us(u-s)}}^{\infty}\frac{e^{-v^2/2}}{\sqrt{2\pi}}dv.
\end{align}
On the event $\{\tau'\leq1/2\}$, we have $1/2\leq u<1$. Combined with the fact that $C>2\pi$, we can derive the following limits as $\varepsilon$ approaches 0 for $y\in (0,\varepsilon^{-1/2}]$,
\begin{equation}\label{largeMubPoissonfirsteventasy1}
\frac{2\beta(\beta-\eta)}{u}<0,\quad \frac{2\beta(\beta-\eta)}{u}=O(\varepsilon^{1/2}),
\end{equation}
\begin{equation}\label{largeMubPoissonfirsteventasy2}
\frac{2\beta s-\eta s-\beta u}{\sqrt{us(u-s)}}<0,\quad \frac{2\beta s-\eta s-\beta u}{\sqrt{us(u-s)}}=O(\varepsilon^{-1/4}),
\end{equation}
\begin{equation}\label{largeMubPoissonfirsteventasy3}
\frac{\beta u-\eta s}{\sqrt{us(u-s)}}>0,\quad \frac{\beta u-\eta s}{\sqrt{us(u-s)}}=O(\varepsilon^{-1/4}).
\end{equation}
Note that none of the asymptotic rates above depend on $y$. Moreover, it can be easily shown that
\begin{equation}\label{largeMubPoissonfirsteventgaussian}
\int_{x}^{\infty} \frac{e^{-v^2/2}}{\sqrt{2\pi}}dv\leq \frac{e^{-x^2/2}}{x\sqrt{2\pi}}.
\end{equation}
By (\ref{largeMubPoissonfirsteventref}), (\ref{largeMubPoissonfirsteventasy1}), (\ref{largeMubPoissonfirsteventasy2}), (\ref{largeMubPoissonfirsteventasy3}) and (\ref{largeMubPoissonfirsteventgaussian}), we see that there exists a positive constant $C_{10}$ such that for $\varepsilon$ sufficiently small, for all $y\in (0,\varepsilon^{-1/2}]$, given $\tau'=1-u\leq1/2$, 
\begin{equation}\label{largeMubPoissonfirsteventub}
P \bigg(\max_{0\leq r\leq \varepsilon}B^{0,u,-\varepsilon^{3/4}y+2C\varepsilon^{1/4}}_{r}\geq \frac{C\varepsilon^{1/4}}{2}\bigg)\leq e^{-C_{10}/\sqrt{\varepsilon}}.
\end{equation}
On the other hand, if $y\in (\varepsilon^{-1/2},\infty)$, given $\tau'=1-u\leq1/2$, by Lemma \ref{stochasticdominanceBM},
\begin{align}\label{largeMubPoissonfirsteventublargey}
P \bigg(\max_{0\leq r\leq \varepsilon}B^{0,u,-\varepsilon^{3/4}y+2C\varepsilon^{1/4}}_{r}\geq \frac{C\varepsilon^{1/4}}{2}\bigg)
&\leq P \bigg(\max_{0\leq r\leq \varepsilon}B^{0,u,(2C-1)\varepsilon^{1/4}}_{r}\geq \frac{C\varepsilon^{1/4}}{2}\bigg)\nonumber\\
&\leq e^{-C_{10}/\sqrt{\varepsilon}}.
\end{align}
As a result, when $z\in (0,2C\varepsilon^{-1/2}]$, taking $C_9<C_{10}$, equation (\ref{largeMubPoissonfirsteventformula}) follows from (\ref{largeMubPoissonfirsteventBM}), (\ref{largeMubPoissonfirsteventcond}), (\ref{largeMubPoissonfirsteventstrong}), (\ref{largeMubPoissonfirsteventub}) and (\ref{largeMubPoissonfirsteventublargey}). 

\qedwhite

\noindent\textit{Proof of Lemma \ref{largeMub}}.
We first transform (\ref{Lem2}) from the setting of branching Brownian motion with absorption and drift into standard branching Brownian motion. Let $P$ be the law of a standard BBM started from $0$ without drift and absorption. We have for $s\in [0,\varepsilon^{-3/2}]$,
\begin{align}
&P^{C/\sqrt{\varepsilon}+\varepsilon s/\rho}_{-\rho}\bigg(\exists u\in \mathcal{N}^{-\rho}_{s}: Y_u(r) >0\; \forall\; r \leq s, \;Y_u(s)> \frac{C}{4\sqrt{\varepsilon}}\bigg)\nonumber\\
&\hspace{0.2in}=P\bigg(\exists u\in \mathcal{N}_{s}: Y_u(r) +\frac{C}{\sqrt{\varepsilon}}+\frac{\varepsilon s}{\rho}-\rho r>0\; \forall r \leq s, \;Y_u(s)+\frac{C}{\sqrt{\varepsilon}}+\frac{\varepsilon s}{\rho}-\rho s > \frac{C}{4\sqrt{\varepsilon}}\bigg)\nonumber\\
&\hspace{0.2in}\geq P\bigg(\exists u\in \mathcal{N}_{s}: Y_u(r)> \rho r -\frac{\varepsilon s}{\rho}-\frac{3C}{4\sqrt{\varepsilon}} \; \forall  r \leq s\bigg).
\label{Lem2smalls}
\end{align}
Then we will apply Theorem 1 in Roberts \cite{roberts2015fine}, which gives the explicit formula of a curve such that at least one particle stays above this curve all the time with nonzero probability. Borrowing notations from \cite{roberts2015fine},  we let $A_c=3^{4/3}\pi^{2/3}2^{-7/6}$ and
\[
g(s)=\sqrt{2}s-A_cs^{1/3}+\frac{A_cs^{1/3}}{\log^2(s+e)}-1.
\]
Theorem 1 in Roberts \cite{roberts2015fine} states that there exists some nonzero absolute constant $C_6$, such that
\[
P(\forall s\geq 0, \exists u\in \mathcal{N}_{s}: Y_u(r)>g(r) \;\forall r\leq s) > C_6.
\]

\noindent Together with our choice of $C>2\pi$ and the Taylor expansion for $\rho$, we have for $\varepsilon$ sufficiently small, for all $r\leq s$,
\[
\rho r-\frac{\varepsilon s}{\rho} -\frac{3C}{4\sqrt{\varepsilon}}
=\sqrt{2}r+\frac{\varepsilon r}{\rho}+O(\varepsilon^2)r-\frac{\varepsilon s}{\rho}-\frac{3C}{4\sqrt{\varepsilon}} \leq \sqrt{2}r-\frac{3C}{4\sqrt{\varepsilon}}+O(\varepsilon^{1/2}),
\]
\[
g(r)\geq \sqrt{2}r-A_c(\varepsilon^{-3/2})^{1/3}-1\geq \sqrt{2}r-\frac{3C}{4\sqrt{\varepsilon}}+O(\varepsilon^{1/2}).
\]
As a result, for all $s\in [0,\varepsilon^{-3/2})$,
\begin{equation}\label{Lem2standardbbm}
P\bigg(\exists u\in \mathcal{N}_{s}: Y_u(r)> \rho r -\frac{\varepsilon s}{\rho}-\frac{3C}{4\sqrt{\varepsilon}} \; \forall  r \leq s\bigg) \geq P\big(\forall s\leq 0, \exists u\in \mathcal{N}_{s}: Y_u(r)>g(r) \;\forall r\leq s\big) > C_6.
\end{equation}
The lemma follows from (\ref{Lem2smalls}) and (\ref{Lem2standardbbm}).
\qedwhite
\\

\noindent\textit{Proof of Lemma \ref{largeMubfinal}.}
From Lemma \ref{largeMub}, we know that if a particle starts from $C/\sqrt{\varepsilon}+\varepsilon s/\rho$, it will have a descendant at time $s$ which stays to the right of $C/(4\sqrt{\varepsilon})$ with probability at least $C_6$. So if we have a particle branching off the spine at a position and time $(t-t_i,\zeta_{t_i})$ satisfying $0\leq t_i\leq \varepsilon^{-3/2}$ and $C/\sqrt{\varepsilon}+\varepsilon t_i/\rho\leq \zeta_{t_i}$, then 
\[
P^{\zeta_{t_i}}_{-\rho}\Big(\exists u\in \mathcal{N}_{t_i}^{\rho}:Y_u(t_i)>\frac{C}{4\sqrt{\varepsilon}}\Big)\geq P_{-\rho}^{C/\sqrt{\varepsilon}+\varepsilon t_i/\rho}\Big(\exists u\in \mathcal{N}_{t_i}^{\rho}:Y_u(t_i)>\frac{C}{4\sqrt{\varepsilon}}\Big)\geq C_6.
\]
Combined with Lemmas \ref{largeMubPoisson} and \ref{largeMub} and the branching property, we have
\begin{align*}
&Q^{y,\varepsilon^{-3/2},z}\bigg(\Big\{\forall u\in\mathcal{N}_{\varepsilon^{-3/2}}^{-\rho}, Y_u(\varepsilon^{-3/2})\leq \frac{C}{4\sqrt{\varepsilon}}\Big\}\cap \{M\geq 2C\sqrt{\varepsilon}\}\bigg)\\
&\hspace{0.1in}\leq Q^{y,\varepsilon^{-3/2},z}\bigg(\{M\geq 2C\sqrt{\varepsilon}\}\cap \bigg\{\sum_{i=1}^{N_\varepsilon}1_{\{\varepsilon \zeta_{t_i}-\frac{1}{\rho}\varepsilon^{2}t_i\geq C\sqrt{\varepsilon}\}}\leq \frac{1}{\sqrt{\varepsilon}}\bigg\}\bigg)\\
&\hspace{0.1in}\quad+Q^{y,\varepsilon^{-3/2},z}\bigg(\forall u\in\mathcal{N}_{\varepsilon^{-3/2}}^{-\rho}, Y_u(\varepsilon^{-3/2})\leq \frac{C}{4\sqrt{\varepsilon}}\Big| \{M\geq 2C\sqrt{\varepsilon}\}\cap\bigg\{\sum_{i=1}^{N_\varepsilon}1_{\{\varepsilon \zeta_{t_i}-\frac{1}{\rho}\varepsilon^{2}t_i\geq C\sqrt{\varepsilon}\}}\geq \frac{1}{\sqrt{\varepsilon}}\bigg\}\bigg)\\
&\hspace{0.1in}\quad\times Q^{y,\varepsilon^{-3/2},z}\Bigg(\bigg\{\sum_{i=1}^{N_\varepsilon}1_{\{\varepsilon \zeta_{t_i}-\frac{1}{\rho}\varepsilon^{2}t_i\geq C\sqrt{\varepsilon}\}}\geq \frac{1}{\sqrt{\varepsilon}}\bigg\}\cap \{M\geq 2C\sqrt{\varepsilon}\}\Bigg)\\
&\hspace{0.1in}\leq \bigg(\frac{2}{yz}+6\bigg)e^{-C_5/\sqrt{\varepsilon}}+(1-C_6)^{1/\sqrt{\varepsilon}}.
\end{align*}
Note that if there exists a $u\in\mathcal{N}_{\varepsilon^{-3/2}}^{-\rho}$ such that $Y_u(\varepsilon^{-3/2})\geq C/(4\sqrt{\varepsilon})$, then for $\varepsilon$ small enough, there exists a $0<C_{11}< C\rho/4$ satisfying
\[
\frac{1}{\sum_{u}Y_u(\varepsilon^{-3/2})e^{\rho Y_u(\varepsilon^{-3/2})}}\leq \frac{4\sqrt{\varepsilon}}{C}e^{-C\rho/(4\sqrt{\varepsilon})}\leq e^{-C_{11}/\sqrt{\varepsilon}}.
\]
As a result,
\begin{align}\label{largeMubeq}
&Q^{y}\Big[\frac{1}{\sum_{u}Y_u(\varepsilon^{-3/2})e^{\rho Y_u(\varepsilon^{-3/2})}}1_{\{M\geq 2C\sqrt{\varepsilon}\}}\Big| \xi_{\varepsilon^{-3/2}}=z\Big]\nonumber\\
&\hspace{0.2in}\leq \frac{1}{ze^{\rho z}}Q^{y,\varepsilon^{-3/2},z}\bigg(\Big\{\forall u\in\mathcal{N}_{\varepsilon^{-3/2}}^{-\rho}, Y_u(\varepsilon^{-3/2})\leq \frac{C}{4\sqrt{\varepsilon}}\Big\}\cap \{M\geq 2C\sqrt{\varepsilon}\}\bigg)\nonumber\\
&\hspace{0.2in}\quad\;+e^{-C_{11}/\sqrt{\varepsilon}}Q^{y,\varepsilon^{-3/2},z}\bigg(\Big\{\exists u\in\mathcal{N}_{\varepsilon^{-3/2}}^{-\rho}, Y_u(\varepsilon^{-3/2})> \frac{C}{4\sqrt{\varepsilon}}\Big\}\cap \{M\geq 2C\sqrt{\varepsilon}\}\bigg)\nonumber\\
&\hspace{0.2in}\leq \frac{2}{yz^2}e^{-C_5\sqrt{\varepsilon}}+\frac{6}{z}e^{-C_5/\sqrt{\varepsilon}}+\frac{1}{z}(1-C_6)^{1/\sqrt{\varepsilon}}+e^{-C_{11}/\sqrt{\varepsilon}}.
\end{align}
Letting $0<C_7<\min\{C_5, -\log(1-C_6), C_{11}\}$, the lemma is proved.
\qedwhite
\\

\noindent\textit{Proof of Lemma \ref{smallMub}.}
First note that if $y\in [(2C+1/\rho)\varepsilon^{-1/2}, \infty)$, then $M\geq2C\sqrt{\varepsilon}$ and  therefore the inequality (\ref{smallMubeq}) holds trivially. It only remains to consider the case where $y\in (0,(2C+1/\rho)\varepsilon^{-1/2})$.

Observe that there is a simple upper bound for (\ref{smallMubeq})
\begin{align}\label{Lem4trivial}
Q^{y}\bigg[\frac{1}{\sum_{u}Y_u(\varepsilon^{-3/2})e^{\rho Y_u(\varepsilon^{-3/2})}}1_{\{M < 2C\sqrt{\varepsilon}\}}\bigg| \xi_{\varepsilon^{-3/2}}=z\bigg]&\leq \frac{1}{z}Q^{y,\varepsilon^{-3/2},z}(M< 2C \sqrt{\varepsilon}).
\end{align}
Furthermore, because $1/\rho<C$,
\begin{align*}
Q^{y,\varepsilon^{-3/2},z}(M< 2C\sqrt{\varepsilon})
&=Q^{y,\varepsilon^{-3/2},z}\Big(\sup_{0\leq s\leq \varepsilon^{-3/2}}(\varepsilon \zeta_s-\frac{1}{\rho}\varepsilon^2 s)< 2C\sqrt{\varepsilon}\Big)\\
&\leq Q^{y,\varepsilon^{-3/2},z}\bigg(\sup_{0\leq s\leq \varepsilon^{-3/2}} \zeta_s< \frac{3C}{\sqrt{\varepsilon}}\bigg)\\
&= Q^{y,\varepsilon^{-3/2},z}\Big(\sup_{0\leq r\leq 1} \varepsilon^{3/4}\zeta_{\varepsilon^{-3/2 }r}< 3C\varepsilon^{1/4}\Big).
\end{align*}
Notice that under $Q^{y,\varepsilon^{-3/2},z}$, the process $\{\varepsilon^{3/4}\zeta_{\varepsilon^{-3/2}r}\}_{0\leq r\leq 1}$ is a Bessel bridge from $\varepsilon^{3/4}z$ to $\varepsilon^{3/4}y$ in time $1$. Recall that $\{X_r^{\varepsilon^{3/4} z,1,\varepsilon^{3/4} y}\}_{0\leq r\leq 1}$ denotes a Bessel bridge from $\varepsilon^{3/4}z$ to $\varepsilon^{3/4}y$ in time 1. For simplicity, below we will omit the superscript of $\{X_r^{\varepsilon^{3/4} z,1,\varepsilon^{3/4} y}\}_{0\leq r\leq 1}$. Therefore,  we have
\begin{equation}\label{Lem4scaling}
Q^{y,\varepsilon^{-3/2},z}(M< 2C\sqrt{\varepsilon})
\leq P\Big(\sup_{0\leq r\leq 1} X_r< 3C\varepsilon^{1/4}\Big).
\end{equation}
According to (0.22) of  \cite{pitman2006combinatorial}, let $B_{(1)}^{0,1,0}, B_{(2)}^{0,1,0}, B_{(3)}^{0,1,0}$ be three independent standard Brownian bridges, we have
\begin{equation*}
X^{0,1,0}\xlongequal{d}\sqrt{\Big(B_{(1)}^{0,1,0}\Big)^2+\Big(B_{(2)}^{0,1,0}\Big)^2+\Big(B_{(3)}^{0,1,0}\Big)^2}.
\end{equation*}
According to Lemma \ref{stochasticdominance} and the above formula, letting $\{X^{0,1,0}_r\}_{0\leq r\leq 1}$ be a Bessel bridge from 0 to 0 in time 1 and $\{B^{0,1,0}_r\}_{0\leq r\leq 1}$ be a Brownian bridge from 0 to 0 in time 1, for $z \in (0,\varepsilon^{-1/2}]$ and $y\in (0,(2C+1/\rho)\varepsilon^{-1/2})$, we get
\begin{equation}\label{Lem4coupling}
P\Big(\sup_{0\leq r\leq 1} X_r< 3C\varepsilon^{1/4}\Big) \leq P\Big(\sup_{0\leq r\leq 1} X^{0,1,0}_r< 3C\varepsilon^{1/4}\Big) 
\leq \bigg[P\Big(\sup_{0\leq r\leq 1} |B_r^{0,1,0}|< 3C\varepsilon^{1/4} \Big)\bigg]^3.
\end{equation}
From Lemma \ref{RBMa=0}, for $\varepsilon$ sufficiently small, 
\begin{equation}\label{Lem4final}
P\Big(\sup_{0\leq r\leq 1}|B^{0,1,0}_r|<3C\varepsilon^{1/4}\Big)\leq C^{-1}\varepsilon^{-1/4}\exp{\Big\{-\frac{\pi^2}{72C^2\sqrt{\varepsilon}}\Big\}}.
\end{equation}
In the end, setting $0<C_8<\pi^2/(24C^2)$, by (\ref{Lem4trivial})--(\ref{Lem4final}), Lemma \ref{smallMub} is proved.
\qedwhite

\section{Lower bound}
\subsection{Proof of the Lower bound}

In this section, we will prove the lower bound (\ref{liminf}). We first state two lemmas, which are the key ingredients in the proof of the lower bound.

We observe that for $\varepsilon$ sufficiently small, the probability that particles which branch off the spine before a large time have descendants at time $t$ is small. As a result, in order to deal with the lower bound, we only need to consider particles that branch off the spine after a large time. We will start by finding this cutoff time $t^*$. 

Let  $0<\delta_1<\delta_2<1/4$. We denote 
\[
t^*:=t-\bigg(\frac{4}{\varepsilon}\bigg)^{2/(1-2\delta_1)}, \quad t':=t-t^{1/2+\delta_2}.
\]
Define $V_{1}$ to be the event that particles that branch off the spine before time $t'$ have descendants alive at time $t$ and the spine stays below $(t')^{1/2+\delta_1}$ for all $s\leq t'$. Define $V_{2}$ to be the event that particles that branch off the spine before time $t'$ have descendants alive at time $t$ and the spine crosses the curve $(t')^{1/2+\delta_1}$ for some $s\leq t'$. Define $V_{3}$ to be the event that particles that branch off the spine between time $t'$ and $t^*$ have descendants alive at time $t$ and the spine stays below the curve $s^{1/2+\delta_1}$ for all $s\in(t', t^*]$. Define $V_{4}$ to be the event that particles that branch off the spine between time $t'$ and $t^*$ have descendants alive at time $t$ and the spine crosses the curve $s^{1/2+\delta_1}$ for some $s\in (t', t^*]$. More precisely, 
\begin{equation}\label{V1def}
V_1 =\{\exists u\in \mathcal{N}_t:O_u\leq t'\}\cap\{\xi_s\leq (t')^{1/2+\delta_1},\;\forall s\leq t' \},
\end{equation}
\begin{equation}\label{V2def}
V_2=\{\exists u\in \mathcal{N}_t:O_u\leq t'\}\cap\{\exists\; s\leq t': \xi_s> (t')^{1/2+\delta_1} \},
\end{equation}
\begin{equation}\label{V3def}
V_3=\{\exists u\in \mathcal{N}_t:t' < O_u\leq t^*\}\cap\{ \xi_s\leq s^{1/2+\delta_1}, \;\forall s\in (t',t^*] \},
\end{equation}
\begin{equation}\label{V4def}
V_4=\{\exists u\in \mathcal{N}_t:t'< O_u\leq t^*\}\cap\{\exists\; s\in (t', t^*]: \xi_s> s^{1/2+\delta_1} \}.
\end{equation}
Then we have,
\[
\{\exists u\in \mathcal{N}_t: O_u\leq t^*\} = V_1\cup V_2\cup V_3\cup V_4.
\]

\begin{Lemma} \label{cutoff}
For any $0<\delta<1/2$, if $\varepsilon$ is sufficiently small, then for all $z\in (0,\varepsilon^{-1/2}]$,
\begin{equation} \label{t*}
\limsup_{t\rightarrow \infty} Q^{x,t,z}\bigg(\bigcup_{i=1}^4 V_i\bigg)< \delta.
\end{equation}
\end{Lemma}
Note that as $\delta_1$ goes to 0, $2/(1-2\delta_1)$ goes to 2. Roughly speaking, this lemma shows that only particles which branch off the spine within the last $\varepsilon^{-2}$ time will contribute significantly to our expectation in Proposition \ref{K}. For simplicity, letting $\kappa=4\delta_1/(1-2\delta_1)>0$, we will also write the cutoff time $t^*$ as
\[
t^*=t-\bigg(\frac{4}{\varepsilon}\bigg)^{2+\kappa}.
\]

We need one more lemma to finish the proof of (\ref{liminf}). Define
\[
M'=\sup_{0\leq s\leq (4/\varepsilon)^{2+\kappa}}\bigg(\varepsilon\zeta_s-\frac{1}{\rho}\varepsilon^2 s\bigg).
\]
Similarly to the proof of upper bound, we will divide the space into two parts, $\{M'\leq C\sqrt{\varepsilon}\}$ and $\{M'> C\sqrt{\varepsilon}\}$ for some constant $C>2\sqrt{3}$. Since this time we focus on the lower bound, it is enough to consider only one of them. 
\begin{Lemma}\label{M'}
Let $C>2\sqrt{3}$. There exists a positive constant $C_{12}$ such that for $\varepsilon$ sufficiently small, for all $z\in (0,\varepsilon^{-1/2}]$ and $y\in (0,\varepsilon^{-1-\kappa}]$, we have
\[
Q^{y,(4/\varepsilon)^{2+\kappa},z}\Big(M'\leq C\sqrt{\varepsilon}\Big)\geq \varepsilon^{-3/4}e^{-C_{12}/\sqrt{\varepsilon}}.
\]
\end{Lemma}

Below, we will apply above lemmas, together with Jensen's inequality and the martingale property to prove the lower bound.\\

\noindent\textit{Proof of (\ref{liminf}).}
Conditioned on the endpoint of the spinal particle, we have 
\begin{align*}
&\sqrt{2\pi t^3}Q^{x}\bigg[\frac{1}{\sum_{u}Y_u(t)e^{\rho Y_u(t)}}\bigg]\nonumber\\
&\hspace{0.2in}\geq \sqrt{2\pi t^3}Q^{x}\bigg[\frac{1}{\sum_{u}Y_u(t)e^{\rho Y_u(t)}};\xi_t\leq \varepsilon^{-1/2}\bigg]\nonumber\\
&\hspace{0.2in}=\sqrt{2\pi t^3}\int_{0}^{\varepsilon^{-1/2}}Q^{x,t,z}\bigg[\frac{1}{\sum_{u}Y_u(t)e^{\rho Y_u(t)}}\bigg]\frac{1}{\sqrt{2\pi t}}\frac{z}{x}e^{-(x-z)^2/(2t)}(1-e^{-2xz/t})dz.
\end{align*}
For every $\varepsilon$ and $x$, there exists a $T(\varepsilon,x)$ such that for all $t\geq T(\varepsilon, x)$ and $z\in (0,\varepsilon^{-1/2}]$,
\begin{equation*}
e^{-(x-z)^2/(2t)} \geq\frac{1}{2}, \quad 1-e^{-2xz/t}\geq \frac{1}{2}\cdot\frac{2xz}{t}=\frac{xz}{t}.
\end{equation*}
Therefore,
\begin{equation}\label{lbendpoint}
\sqrt{2\pi t^3}Q^{x}\bigg[\frac{1}{\sum_{u}Y_u(t)e^{\rho Y_u(t)}}\bigg]\geq \frac{1}{2}\int_{0}^{\varepsilon^{-1/2}}Q^{x,t,z}\bigg[\frac{1}{\sum_{u}Y_u(t)e^{\rho Y_u(t)}}\bigg]z^2dz.
\end{equation}

Next, we restrict the integrand to the case where all particles branch off the spine after $t^*$. Letting
\[ 
V=\big\{\forall u \in \mathcal{N}_t: O_u> t^*\big\}=\bigcap_{i=1}^{4}V_i^c,
\]
we have
\begin{equation}\label{lbtruncation}
Q^{x,t,z}\bigg[\frac{1}{\sum_{u}Y_u(t)e^{\rho Y_u(t)}}\bigg]\geq Q^{x,t,z}\bigg[\frac{1_{V}}{\sum_{u}Y_u(t)e^{\rho Y_u(t)}}\bigg]=Q^{x,t,z}\bigg[\frac{1_{V}}{\sum_{u}Y_u(t)e^{\rho Y_u(t)}1_{\{O_u> t^*\}}}\bigg].
\end{equation}
Define $\mathcal{G}_{t}$ to be the $\sigma-$field generated by $V$ and the whole trajectory of the spine, $\{\xi_s\}_{0\leq s\leq t}$. In other words, $\mathcal{G}_t$ contains all the information regarding the movement of the spine and the event that all descendants alive at time $t$ branch off the spine after $t^*$. Conditioning on $\mathcal{G}_{t}$, Jensen's inequality for conditional expectation gives
\begin{align}\label{lbJensen}
Q^{x,t,z}\bigg[\frac{1_V}{\sum_{u}Y_u(t)e^{\rho Y_u(t)}1_{\{O_u> t^*\}}}\bigg]
&\geq Q^{x,t,z}\bigg[\frac{1_{V\cap \{M'\leq C\sqrt{\varepsilon}\}}}{\sum_{u}Y_u(t)e^{\rho Y_u(t)}1_{\{O_u> t^*\}}}\bigg]\nonumber\\
&= Q^{x,t,z}\bigg[1_{V\cap\{M'\leq C\sqrt{\varepsilon}\}}Q^{x,t,z}\bigg[\frac{1}{\sum_{u}Y_u(t)e^{\rho Y_u(t)}1_{\{O_u> t^*\}} } \bigg|\mathcal{G}_{t} \bigg]\bigg]\nonumber\\
&\geq Q^{x,t,z}\bigg[\frac{1_{V\cap\{M'\leq C\sqrt{\varepsilon}\}}}{Q^{x,t,z}\big[\sum_{u}Y_u(t)e^{\rho Y_u(t)}1_{\{O_u> t^*\}}\big| \mathcal{G}_{t}\big]}\bigg].
\end{align}
To deal with the denominator, we need to use the fact that for every $\varepsilon$, $\big\{\sum_u Y_u(t)e^{\rho Y_u(t)+\varepsilon t}\big\}_{t\geq 0}$ is a martingale for the original BBM with absorption. Under the measure $Q$, particles branch off the spine with rate $2$ and initiate independent copies of the original BBM with absorption. Note that $\xi_s=\zeta_{t-s}$ for $0\leq s\leq t$.
Then by the spinal decomposition and the formula for expectations of additive functionals of Poisson point processes, we have
\begin{align}\label{lbspinepoisson}
Q^{x,t,z}\bigg[\sum_{u}Y_u(t)e^{\rho Y_u(t)}1_{\{O_u> t^*\}}\Big|\mathcal{G}_{t}\bigg]
&=2\int_{t^*}^{t} \xi_{r}e^{\rho\xi_{r}-\varepsilon(t-r)}dr+ze^{\rho z}\nonumber\\
&=2\int_{0}^{(4/\varepsilon)^{2+\kappa}}\zeta_se^{\rho\zeta_s-\varepsilon s}ds+ze^{\rho z}.
\end{align}
Moreover, on the event $\{M'\leq C\sqrt{\varepsilon}\}$, if $\varepsilon$ is sufficiently small, for all $0\leq s\leq (4/\varepsilon)^{2+\kappa}$,
\[
\zeta_s\leq \frac{1}{\rho}\bigg(\frac{4}{\varepsilon}\bigg)^{2+\kappa}\varepsilon+\frac{C}{\sqrt{\varepsilon}}\leq 4^{2+\kappa}\varepsilon^{-1-\kappa},
\]
and
\[
\rho \zeta_s-\varepsilon s= \frac{\rho}{\varepsilon}(\varepsilon \zeta_s-\frac{1}{\rho}\varepsilon^2s)\leq \frac{\rho}{\varepsilon}C\sqrt{\varepsilon}=\frac{C\rho}{\sqrt{\varepsilon}}.
\]
Thus, when $M' \leq C\sqrt{\varepsilon}$, for all $z\in (0,\varepsilon^{-1/2}]$,
\begin{equation}\label{lbM}
2\int_{0}^{(4/\varepsilon)^{2+\kappa}}\zeta_se^{\rho\zeta_s-\varepsilon s}ds+ze^{\rho z} \leq 2\bigg(\frac{4}{\varepsilon}\bigg)^{2+\kappa}\cdot 4^{2+\kappa}\varepsilon^{-1-\kappa}\cdot e^{C\rho/ \sqrt{\varepsilon}}+\varepsilon^{-1/2}e^{\rho/\sqrt{\varepsilon}}\leq 2^{10+4\kappa}\varepsilon^{-3-2\kappa}e^{C\rho/\sqrt{\varepsilon}}.
\end{equation}
Combining (\ref{lbJensen}), (\ref{lbspinepoisson}) and (\ref{lbM}), we have
\begin{equation}\label{lbprobability}
Q^{x,t,z}\bigg[\frac{1_V}{\sum_{u}Y_u(t)e^{\rho Y_u(t)}1_{\{O_u> t^*\}}}\bigg] \geq 2^{-10-4\kappa}\varepsilon^{3+2\kappa}e^{-C\rho/\sqrt{\varepsilon}}Q^{x,t,z}\Big(V\cap \{M'\leq C\sqrt{\varepsilon}\}\Big).
\end{equation}

It remains to find a lower bound for the probability of the above event. Because $\{\xi_s\}_{0\leq s\leq t}$ is a Markov process under $Q^{x,t,z}$, we have $\{\xi_s\}_{0\leq s\leq t^*}$ is conditionally independent of $\{\xi_s\}_{t^*\leq s\leq t}$ given $\xi_{t^*}$. Furthermore, note that $V$ is the event that particles which branch off the spine before time $t^*$ all become extinct before time $t$ and once a particle branches off the spine, it initiates a BBM independent of the future trajectory of the spine. As a result, conditioned on $\xi_{t^*}$, the events $V$ and $\{M'\leq C\sqrt{\varepsilon}\}$ are independent. By Lemma \ref{M'}, we obtain
\begin{align}\label{lbprobabilityestimate}
Q^{x,t,z}\Big(V\cap \{M'\leq C\sqrt{\varepsilon}\}\Big)&=\int_{0}^{\infty}Q^{x,t,z}\Big(V\cap \{M'\leq C\sqrt{\varepsilon}\}\Big|\xi_{t^*}=y\Big)p_{t^*}^{x,t,z}(y) dy\nonumber\\
&=\int_{0}^{\infty} Q^{x,t,z}\big(V\big|\xi_{t^*}=y\big)Q^{x,t,z}\big(M'\leq C\sqrt{\varepsilon}\big|\xi_{t^*}=y\big)p_{t^*}^{x,t,z}(y) dy\nonumber\\
&\geq \int_{0}^{\varepsilon^{-1-\kappa}}Q^{x,t,z}\big(V\big|\xi_{t^*}=y\big)Q^{y,(4/\varepsilon)^{2+\kappa},z}\big( M'\leq C\sqrt{\varepsilon}\big)p_{t^*}^{x,t,z}(y) dy\nonumber\\
&\geq \varepsilon^{-3/4}e^{-C_{12}/\sqrt{\varepsilon}}Q^{x,t,z}\Big(V\cap \big\{\xi_{t^*}\leq \varepsilon^{-1-\kappa}\big\}\Big)\nonumber\\
&= \varepsilon^{-3/4}e^{-C_{12}/\sqrt{\varepsilon}}\Big[Q^{x,t,z}\Big(\xi_{t^*}\leq \varepsilon^{-1-\kappa}\Big)-Q^{x,t,z}\Big(V^c\cap \big\{\xi_{t^*}\leq \varepsilon^{-1-\kappa}\big\}\Big)\Big].
\end{align}
As for the first term, note that $\{\xi_{t^*}\leq \varepsilon^{-1-\kappa}\}=\{\zeta_{(4/\varepsilon)^{2+\kappa}}\leq \varepsilon^{-1-\kappa}\}$, where $\{\zeta_{s}\}_{0\leq s\leq t}$ is a Bessel bridge from $z$ to $x$ in time $t$ under $Q^{x,t,z}$. Define $\{R_r^{z}\}_{r\geq 0}$ to be a Bessel process starting from $z$. We apply Lemma \ref{weakconvergence} to obtain,
\[
\lim_{t\rightarrow \infty}Q^{x,t,z}\Big(\xi_{t^*}\leq \varepsilon^{-1-\kappa}\Big)=P \Big(R^{z}_{(4/\varepsilon)^{2+\kappa}}\leq \varepsilon^{-1-\kappa}\Big).
\]
According to the scaling property of the Bessel process, we have for $\varepsilon$ sufficiently small, for all $z\in (0,\varepsilon^{-1/2}]$,
\begin{align*}
P \Big(R^{z}_{(4/\varepsilon)^{2+\kappa}}\leq \varepsilon^{-1-\kappa}\Big)&
=P\bigg(\Big(\frac{\varepsilon}{4}\Big)^{1+\kappa/2}R^{z}_{(4/\varepsilon)^{2+\kappa}}\leq \Big(\frac{\varepsilon}{4}\Big)^{1+\kappa/2}\varepsilon^{-1-\kappa}\bigg)\\
&=P \bigg(R^{z(\varepsilon/4)^{1+\kappa/2}}_{1}\leq \frac{\varepsilon^{-\kappa/2}}{4^{1+\kappa/2}}\bigg)\\
&> \frac{1}{2}.
\end{align*}
Therefore, for $\varepsilon$ small enough,  for all $z\in (0,\varepsilon^{-1/2}]$, if $t$ is large enough, we have
\begin{equation}\label{lbprobabilityestimate1}
Q^{x,t,z}\Big(\xi_{t^*}\leq \varepsilon^{-1-\kappa}\Big)\geq \frac{1}{2}.
\end{equation}
As for the second term, according to Lemma \ref{cutoff}, for $\varepsilon$ sufficiently small, for all $z\in (0,\varepsilon^{-1/2}]$, if $t$ is large enough, then
\begin{equation}\label{lbprobabilityestimate2}
Q^{x,t,z}\Big(V^c\cap \big\{\xi_{t^*}\leq \varepsilon^{-1-\kappa}\big\}\Big)\leq Q^{x,t,z}\Big(V^c\Big)<\frac{1}{4}.
\end{equation}

In the end, by (\ref{lbendpoint}), (\ref{lbtruncation}), (\ref{lbprobability})--(\ref{lbprobabilityestimate2}) and Fatou's Lemma, we proved that for $\varepsilon$ small enough,
\begin{align*}
&\liminf_{t\rightarrow \infty}\sqrt{2\pi t^3}Q^{x}\bigg[\frac{1}{\sum_{u}Y_u(t)e^{\rho Y_u(t)}}\bigg] \\
&\hspace{0.2in}\geq \frac{1}{2}\int_{0}^{\varepsilon^{-1/2}}2^{-10-4\kappa}\varepsilon^{2+2\kappa}e^{-C\rho/\sqrt{\varepsilon}}\liminf_{t\rightarrow \infty}Q^{x,t,z}\Big(V\cap\{M'\leq C\sqrt{\varepsilon}\}\Big)z^2dz \\
&\hspace{0.2in}\geq 2^{-13-4\kappa}3^{-1}\varepsilon^{3/4+2\kappa}e^{-(C\rho+C_{12})/\sqrt{\varepsilon}}.
\end{align*}
Consequently, the lower bound in Theorem \ref{main} is proved as long as 
\[
C_2>2C+C_{12}\geq C\rho+C_{12}.
\]
\qedwhite

\subsection{Proof of Lemmas}
Before proving Lemma \ref{cutoff}, we need one more ingredient. Recall that $\{R_r^{z}\}_{z\geq 0}$ is a Bessel process starting from $z$.
\begin{Lemma} \label{appendix}
 For every fixed $\varepsilon$, we have
\[
\lim_{t\rightarrow\infty}Q^{x,t,z}\bigg(\exists r\geq  \Big(\frac{4}{\varepsilon}\Big)^{2/(1-2\delta_1)} : \zeta_r\geq r^{1/2+\delta_1}\bigg)=P\bigg(\exists r\geq  \Big(\frac{4}{\varepsilon}\Big)^{2/(1-2\delta_1)} : R_r^z\geq r^{1/2+\delta_1}\bigg).
\]
\end{Lemma}
\begin{proof}
According to Lemma \ref{weakconvergence}, the Bessel bridge converges to the Bessel process in the Skorokhod topology. Recall that under $Q^{x,t,z}$, the process $\{\zeta_r\}_{0\leq r\leq t}$ is a Bessel bridge from $z$ to $x$ in time $t$. Since both Bessel bridges and the Bessel process are continuous, the Skorokhod topology in this case coincides with the uniform topology. Thus, it is sufficient to prove that for a Bessel process $\{R_r^{z}\}_{r\geq 0}$ starting from $z$ under $P$, for every constant $c\geq 1$, the event 
\[A:=\Big\{\exists r\geq c: R_r^{z}\geq r^{1/2+\delta_1}\Big\}\]
is a continuity set under the uniform topology. That is to say, letting $\partial A$ denote the boundary set of $A$ under the uniform topology, essentially, we want to prove that
\begin{equation}\label{simplifiedgoal}
P(\partial A)=P\big(\{\omega:\{R_r^{z}(\omega)\}_{r\geq 0}\in \partial A\}\big)=0.
\end{equation}

We first consider elements in $\partial A$, which can be approached both from $A$ and $A^c$ under the uniform topology. Note that $A^c=\{\forall r\geq c, R_r^{z}<r^{1/2+\delta_1}\}$.
For $\{R_r^{z}(\omega)\}_{r\geq 0} \in \partial A$, if there exists an $r\geq c$ such that $R_r^{z}(\omega)>r^{1/2+\delta_1}$, then $\{R_r^{z}(\omega)\}_{r\geq 0}$ cannot be approached from $A^c$. As a result, $R_r^{z}(\omega)\leq r^{1/2+\delta_1}$ for all $r\geq c$. Furthermore, if $\inf_{r\geq c}(r^{1/2+\delta_1}-R_r^{z}(\omega))>0$, then it cannot be approached from $A$. Thus, $\inf_{r\geq c}(r^{1/2+\delta_1}-R_r^{z}(\omega))=0$. Indeed, this infimum must be attained at some finite value of $r$ because of the law of iterated logarithm of the Bessel process at infinity (see, e.g., IV.40 of \cite{borodin2012handbook}). More precisely, letting $\sigma=\inf\{r\geq c: R_r^{z}=r^{1/2+\delta_1}\}$, we see that
\begin{align*}
P(\{\sigma=\infty\}\cap \partial A)&
\leq P\Big(\{\sigma=\infty\}\cap\big\{\inf_{r\geq c}(r^{1/2+\delta_1}-R_r^z)=0\big\}\Big)\\
&\leq P\Big(\lim_{r\rightarrow \infty}\big(r^{1/2+\delta_1}-R_r^z\big)=0\Big)\\
&=0.
\end{align*}
Note that $\sigma$ is a stopping time and let $\mathcal{F}_{\sigma}$ be the $\sigma-$ field generated by $\sigma$. By the strong Markov property of the Bessel process, we have
\begin{align}\label{boundary}
P(\partial A)&=P(\partial A\cap \{\sigma<\infty\})\nonumber\\
&\leq P\Big(\Big\{\forall r\geq c: R_r^{z}\leq r^{1/2+\delta_1}\Big\}\cap\{\sigma<\infty\}\Big)\nonumber\\
&=E\Big[1_{\{\sigma<\infty\}}P\Big(\forall r\geq 0, R_r^{\sigma^{1/2+\delta_1}}\leq (r+\sigma)^{1/2+\delta_1}\Big)\Big].
\end{align}
Using the same method as the proof of Lemma \ref{stochasticdominance}, it can be shown that the Bessel process $\{R_r^z\}_{r\geq 0}$ stochastically dominates Brownian motion $\{B_r^{z}\}_{r\geq 0}$. Thus, conditioned on $\mathcal{F}_\sigma$,
\begin{align}\label{translation}
P\Big(\forall r\geq 0, R^{\sigma^{1/2+\delta_1}}_r\leq (r+\sigma)^{1/2+\delta_1}\Big)
&\leq P\Big(\forall r\geq 0, B^{\sigma^{1/2+\delta_1}}_r\leq (r+\sigma)^{1/2+\delta_1}\Big)\nonumber\\
&=P\Big(\forall r\geq 0, B_r\leq (r+\sigma)^{1/2+\delta_1}-\sigma^{1/2+\delta_1}\Big).
\end{align}
By the law of the iterated logarithm at $0$ for Brownian motion (see, e.g., IV.5 of \cite{borodin2012handbook}), we have almost surely
\begin{equation}\label{LIL}
\limsup_{t\rightarrow 0}\frac{B_t}{\sqrt{2t\ln\ln(1/t)}}=1.
\end{equation}
Conditioned on $\mathcal{F}_{\sigma}$, since $0<\delta_1<1/4$ and $\sigma\geq c\geq 1$, a Taylor expansion gives  
\[
(r+\sigma)^{1/2+\delta_1}-\sigma^{1/2+\delta_1}=\bigg(\frac{1}{2}+\delta_1\bigg)\sigma^{-1/2+\delta_1}r+o(r)\leq \bigg(\frac{1}{2}+\delta_1\bigg)r+o(r).
\]
Since $0<\delta_1<1/4$, conditioned on $\mathcal{F}_{\sigma}$, there exists an $\alpha$ such that for all $r\leq \alpha$,
\begin{equation}\label{small}
(r+\sigma)^{1/2+\delta_1}-\sigma^{1/2+\delta_1}<\frac{3}{4}\sqrt{2r\ln\ln\bigg(\frac{1}{r}\bigg)}.
\end{equation}
From (\ref{LIL}) and (\ref{small}), conditioned on $\mathcal{F}_{\sigma}$, we have
\begin{equation}\label{translation0prob}
P(\forall r\geq 0, B_r\leq (r+\sigma)^{1/2+\delta_1}-\sigma^{1/2+\delta_1})=0.
\end{equation}
By (\ref{boundary}), (\ref{translation}) and (\ref{translation0prob}), equation (\ref{simplifiedgoal}) is proved and the lemma follows.
\end{proof}

\noindent\textit{Proof of Lemma \ref{cutoff}.}
Let $u(t,x)$ be the probability of survival at time $t$ for a branching Brownian motion starting from $x$ under $P^x_{-\rho}$. It is pointed out in equation (5) of \cite{harris2007survival} that
\begin{equation}\label{u}
u(t,x)\leq e^{\rho x-\varepsilon t}.
\end{equation}
Moreover, we write $0\leq \tau_1<\tau_2<...\leq t$ for the successive branching times along the spine. Note that under $Q^x$, particles branch off the spine at rate $2$. We define $\mathcal{N}^i_t$ to be the the set of surviving particles at time $t$ which have branched off the spine at time $\tau_i$.  Inheriting notations from Section 3, we denote by $p_s(x,y)$ the transition probability of a Bessel process and $p^{x,t,z}_s(y)$ the transition probability of a Bessel bridge from $x$ to $z$ within time $t$. 

Start with $V_1$ which is defined in (\ref{V1def}). Applying (\ref{u}), we have
\begin{align}\label{V_1}
Q^{x,t,z}(V_{1}) 
& \leq Q^{x,t,z}\Bigg[\sum_{i:\tau_i\leq t'}1_{\{\mathcal{N}_t^i \neq\emptyset\}}1_{\{\xi_s\leq (t')^{1/2+\delta_1},\; \forall s\leq t'\}}\Bigg]\nonumber\\
& \leq 2\int_{0}^{t'} \int_{0}^{(t')^{1/2 +\delta_1}} u(t-s,y) p_{s}^{x,t,z}(y) dy ds\nonumber\\
& \leq 2\int_{0}^{t'} \int_{0}^{(t')^{1/2 +\delta_1}} e^{\rho y -\varepsilon (t-s)} p_{s}^{x,t,z}(y)  dy ds\nonumber\\
& \leq 2e^{\rho (t')^{1/2+\delta_1}} \int_{0}^{t'}  e^{-\varepsilon (t-s)} ds\nonumber\\
& \leq \frac{2}{\varepsilon}e^{-\varepsilon(t-t')+\rho (t')^{1/2+\delta_1}}.
\end{align}
For every fixed $\varepsilon$, since $0<\delta_1<\delta_2<1/4$, we have
\[
\lim_{t\rightarrow \infty}\frac{\varepsilon(t-t')}{\rho (t')^{1/2 +\delta_1}}= \infty.
\] 
Therefore, for every fixed $\varepsilon$,
\begin{equation}\label{V1}
\lim_{t\rightarrow \infty}Q^{x,t,z}(V_1)=0.
\end{equation}

As for $V_2$ and $V_4$, which are defined in equations (\ref{V2def}) and (\ref{V4def}) respectively, notice that the process $\{\xi_{s}\}_{0\leq s\leq t}$ is a Bessel bridge from $x$ to $z$ in time $t$ under $Q^{x,t,z}$. According to the scaling property of the Bessel bridge, $\{\xi_{rt}/\sqrt{t}\}_{0\leq r\leq 1}$ is a Bessel bridge from $x/\sqrt{t}$ to $z/\sqrt{t}$ within time 1. Define $\{X_r^{x/\sqrt{t},1,z/\sqrt{t}}\}_{0\leq r\leq 1}$ to be a Bessel bridge from $x/\sqrt{t}$ to $z/\sqrt{t}$ within time 1. Then we have
\[
0\leq Q^{x,t,z}(V_2)\leq Q^{x,t,z}\Big(\exists s\leq t', \xi_s>(t')^{1/2+\delta_1}\Big)=P\bigg(\exists r\leq \frac{t'}{t}, X_r^{x/\sqrt{t},1,z/\sqrt{t}}>\frac{(t')^{1/2+\delta_1}}{t^{1/2}}\bigg).
\]
Note that 
\[
\frac{(t')^{1/2+\delta_1}}{t^{1/2}}=\frac{(t-t^{1/2+\delta_2})^{1/2+\delta_1}}{t^{1/2}}\rightarrow \infty, \quad \mbox{as}\; t\rightarrow \infty.
\]
Accordingly, for every fixed $\varepsilon$,
\begin{equation}\label{V2}
0\leq\lim_{t\rightarrow\infty}Q^{x,t,z}(V_2)\leq \lim_{t\rightarrow\infty}P\bigg(\exists r<1, X^{x/\sqrt{t},1,z/\sqrt{t}}_r>\frac{(t')^{1/2+\delta_1}}{t^{1/2}}\bigg)=0.
\end{equation}
Similarly for $V_4$,
\[
Q^{x,t,z}(V_4)\leq Q^{x,t,z}\Big(\exists s\in(t', t^*], \xi_s>s^{1/2+\delta_1}\Big)=P\bigg(\exists r\in\bigg(\frac{t'}{t},\frac{t^*}{t}\bigg], X^{x/\sqrt{t},1,z/\sqrt{t}}_r>\frac{(rt)^{1/2+\delta_1}}{t^{1/2}}\bigg).
\]
Because as $t\rightarrow \infty$, for all $r\in (t'/t,t^*/t]$, $(rt)^{1/2+\delta_1}/t^{1/2}\rightarrow \infty$, we have for fixed $x$ and $\varepsilon$, for all $z\in(0,\varepsilon^{-1/2}]$,
\begin{equation}\label{V4}
0\leq \lim_{t\rightarrow \infty}Q^{x,t,z}(V_4)\leq \lim_{t\rightarrow\infty}P\bigg(\exists r\in\bigg(\frac{t'}{t}, \frac{t^*}{t}\bigg], X^{x/\sqrt{t},1,z/\sqrt{t}}_r>\frac{(rt)^{1/2+\delta_1}}{t^{1/2}}\bigg)=0.
\end{equation}

It remains to work on $Q^{x,t,z}(V_3)$. Recall that $V_3$ is defined in (\ref{V3def}). We will separate $Q^{x,t,z}(V_3)$ into two parts and show both of them have small probability as $t\rightarrow \infty$. For $z\in (0,\varepsilon^{-1/2}]$,
\begin{align}\label{V3separate}
Q^{x,t,z}(V_3)&=Q^{x,t,z}\Big(\{\exists u\in \mathcal{N}_t:t' < O_u\leq t^*\}\cap\{\xi_s\leq s^{1/2+\delta_1}, \;\forall s\in (t', t^*] \}\Big)\nonumber\\
&\leq Q^{x,t,z}\Big(\{\exists u\in \mathcal{N}_t:t' < O_u\leq t^*\}\cap\{ \xi_s\leq (t-s)^{1/2+\delta_1},\;\forall s\in(t', t^*] \}\Big)\nonumber\\
&\quad +Q^{x,t,z}\Big(\exists s\in (t', t^*]: \xi_s\geq (t-s)^{1/2+\delta_1}\Big)\nonumber\\
&=:H_1+H_2.
\end{align}
For $H_1$, we have
\begin{align*}
H_1
& \leq Q^{x,t,z} \Bigg[\sum_{i:t'<\tau_t\leq t^*}1_{\{\mathcal{N}^i_t\neq\emptyset\}}1_{\{\xi_s\leq (t-s)^{1/2+\delta_1}, \;\forall s\in(t', t^*]\}}\Bigg]\\
& \leq 2\int_{t'}^{t^*} \int_{0}^{(t-s)^{1/2+\delta_1}} u(t-s,y) p^{x,t,z}_{s}(y)dy ds.
\end{align*}
Letting $r=t-s$ and noting that a time-reversed Bessel bridge is also a Bessel bridge, we get
\begin{align}\label{J1first}
H_1
&\leq 2\int_{t-t^*}^{t^{1/2+\delta_2}}\int_{0}^{r^{1/2+\delta_1}}u(r,y)p^{z,t,x}_{r}(y)dydr\nonumber\\
& \leq 2\int_{t-t^*}^{t^{1/2+\delta_2}}\int_{0}^{r^{1/2+\delta_1}}e^{\rho y-\varepsilon r}p_{r}(z,y)\frac{p_{t-r}(y,x)}{p_t(z,x)}dydr.
\end{align}
We further observe that for large time $t$, the difference between the probability density functions of the Bessel bridge and the Bessel process are negligible. Note that $1-e^{-x}\leq x$ for $x\geq 0$ and $1-e^{-x}\geq x/2$ for $0\leq x\leq 1$. Then for every fixed $\varepsilon$, when $t$ is large enough, we have for all $z\in (0,\varepsilon^{-1/2}]$, $y\in (0,r^{1/2+\delta_1}]$, and $r\in [t-t^*,t^{1/2+\delta_2}]$ uniformly, 
\begin{equation*}
\frac{p_{t-r}(y,x)}{p_t(z,x)} 
=\frac{\frac{1}{\sqrt{2\pi(t-r)}}\cdot \frac{x}{y}e^{-(y-x)^2/2(t-r)}(1-e^{-2yx/(t-r)})}{\frac{1}{\sqrt{2\pi t}}\cdot \frac{x}{z}e^{-(z-x)^2/2t}(1-e^{-2xz/t})} \leq \sqrt{\frac{t}{t-r}}\frac{\frac{x}{y}\cdot \frac{2yx}{t-r}}{\frac{x}{z}\cdot \frac{2}{3}\cdot \frac{xz}{t}}=3\bigg(\frac{t}{t-r}\bigg)^{3/2} \leq 4.
\end{equation*}
Also see that for $r\geq t-t^*$, $\rho r^{1/2+\delta_1}\leq \rho\varepsilon r/4$. 
Based on the above two observations, we have for sufficiently small $\varepsilon$, if $t$ is large enough, then
\begin{align}\label{J1second}
2\int_{t-t^*}^{t^{1/2+\delta_2}}\int_{0}^{r^{1/2+\delta_1}}e^{\rho y-\varepsilon r}p_{r}(z,y)\frac{p_{t-r}(y,x)}{p_t(z,x)}dydr\nonumber
&\leq 8\int_{t-t^*}^{t^{1/2+\delta_2}}\int_{0}^{r^{1/2+\delta_1}}e^{\rho y-\varepsilon r}p_{r}(z,y)dydr\nonumber\\
&\leq 8\int_{t-t^*}^{t^{1/2+\delta_2}}e^{\rho r^{1/2+\delta_1}-\varepsilon r}dr\nonumber\\
&\leq 8\int_{t-t^*}^{t^{1/2+\delta_2}}e^{(\rho\varepsilon r/4)-\varepsilon r}dr\nonumber\\
&\leq \frac{8}{\varepsilon(1-\rho/4)}e^{-\varepsilon(1-\rho/4)(4/\varepsilon)^{2/(1-2\delta_1)}}.
\end{align}
Since $0<\delta_1<1/4$, together with (\ref{J1first}) and (\ref{J1second}), we have for any $0<\delta<1$, if $\varepsilon$ is sufficiently small, 
\begin{equation}\label{J1}
\limsup_{t\rightarrow \infty}H_1<\frac{16}{\varepsilon}e^{-1/\varepsilon}<\frac{\delta}{2}.
\end{equation}
As for $H_2$, we will apply the law of the iterated logarithm for the Bessel process (see, e.g., IV.40 of \cite{borodin2012handbook}) for all $z\in (0,\varepsilon^{-1/2}]$,
\[
P\bigg(\limsup_{t\rightarrow \infty} \frac{R_t^{z}}{\sqrt{2t\ln\ln t}}=1\bigg)=1.
\] 
Then it follows that for all $z\in (0,\varepsilon^{-1/2}]$,
\begin{equation}\label{QBessel}
\lim_{t\rightarrow \infty}P\Big( R_s^{z}< s^{1/2+\delta_1}, \; \forall s\geq t\Big)=1.
\end{equation}
Recall that $\{\zeta_s\}_{0\leq s\leq t}$ denotes the time-reversed Bessel bridge, which is a Bessel bridge from $z$ to $x$ in time $t$ under $Q^{x,t,z}$. From Lemma \ref{appendix} and (\ref{QBessel}), if $\varepsilon$ is sufficiently small, we have
\begin{align}\label{J2}
\limsup_{t\rightarrow\infty}H_2&=\limsup_{t\rightarrow \infty}Q^{x,t,z}\Big(\exists\; t\in [t',t^*]: \xi_s\geq (t-s)^{1/2+\delta_1}\Big)\nonumber\\
&\leq\limsup_{t\rightarrow \infty}Q^{x,t,z}\bigg(\exists\;r>\Big(\frac{4}{\varepsilon}\Big)^{2/(1-2\delta_1)}:\zeta_r\geq r^{1/2+\delta_1}\bigg)\nonumber\\
&=P\bigg(\exists\; r>\Big(\frac{4}{\varepsilon}\Big)^{2/(1-2\delta_1)}: R_r^z\geq r^{1/2+\delta_1}\bigg)\nonumber\\
&<\frac{\delta}{2}.
\end{align}
Consequently, by (\ref{V3separate}), (\ref{J1}) and (\ref{J2}), for sufficiently small $\varepsilon$,
\begin{equation}\label{V3}
\limsup_{t\rightarrow \infty} Q^{x,t,z}(V_3)<\delta.
\end{equation}
Together with (\ref{V1}), (\ref{V2}) and (\ref{V4}), the lemma follows.
\qedwhite
\\

\noindent\textit{Proof of Lemma \ref{M'}.}
Under $Q^{y,(4/\varepsilon)^{2+\kappa},z}$, the reversed trajectory of the spine $\{\zeta_{s}\}_{0\leq s \leq (4/\varepsilon)^{2+\kappa}}$ is a Bessel bridge from $z$ to $y$ within time $(4/\varepsilon)^{2+\kappa}$. After scaling, $\big\{(\varepsilon/4)^{1+\kappa/2}\zeta_{(4/\varepsilon)^{2+\kappa}r}\big\}_{0\leq r\leq 1}$ is a Bessel bridge from $(\varepsilon/4)^{1+\kappa/2} z$ to $(\varepsilon/4)^{1+\kappa/2} y$ within time $1$. Recall that the process $\{X_r^{(\varepsilon/4)^{1+\kappa/2}z,1,(\varepsilon/4)^{1+\kappa/2}y}\}_{0\leq r\leq 1}$ is a Bessel bridge from $(\varepsilon/4)^{1+\kappa/2}z$ to $(\varepsilon/4)^{1+\kappa/2}y$ in time $1$. For simplicity, we will write $\{X_r\}_{0\leq r\leq 1}$ in place of $\{X_r^{(\varepsilon/4)^{1+\kappa/2}z,1,(\varepsilon/4)^{1+\kappa/2}y}\}_{0\leq r\leq 1}$. Accordingly, we have
\begin{align*}
&Q^{y,(4/\varepsilon)^{2+\kappa},z}\Big(M'\leq C\sqrt{\varepsilon}\Big)\\
&\hspace{0.2in}=Q^{y,(4/\varepsilon)^{2+\kappa},z}\bigg(\sup_{0\leq s\leq (4/\varepsilon)^{2+\kappa}} \bigg(\varepsilon \zeta_s-\frac{1}{\rho}\varepsilon^2s\bigg)\leq C\sqrt{\varepsilon}\bigg)\\
&\hspace{0.2in}=Q^{y,(4/\varepsilon)^{2+\kappa},z}\bigg(\sup_{0\leq r\leq 1}\bigg(\Big(\frac{\varepsilon}{4}\Big)^{1+\kappa/2}\zeta_{(4/\varepsilon)^{2+\kappa}r}-\frac{4^{1+\kappa/2}}{\rho}\varepsilon^{-\kappa/2}r\bigg)\leq \frac{C}{4^{1+\kappa/2}}\varepsilon^{(\kappa+1)/2}\bigg)\\
&\hspace{0.2in}=P\bigg(\sup_{0\leq r\leq 1}\bigg(X_r-\frac{4^{1+\kappa/2}}{\rho}\varepsilon^{-\kappa/2}r\bigg)\leq \frac{C}{4^{1+\kappa/2}}\varepsilon^{(\kappa+1)/2}\bigg).
\end{align*}
By (0.22) of  \cite{pitman2006combinatorial}, we can represent $\{X_r\}_{0\leq r\leq 1}$ in terms of three independent standard Brownian bridges, $B_{(1)}^{0,1,0}, B_{(2)}^{0,1,0}, B_{(3)}^{0,1,0}$,
\[
X_r\xlongequal{d}\sqrt{\bigg(\Big(\frac{\varepsilon}{4}\Big)^{1+\kappa/2}z(1-r)+\Big(\frac{\varepsilon}{4}\Big)^{1+\kappa/2}yr+B_{(1)}^{0,1,0}(r)\bigg)^2+\Big(B_{(2)}^{0,1,0}(r)\Big)^2+\Big(B_{(3)}^{0,1,0}(r)\Big)^2}.
\]
According to the two formulas above, because $C>2\sqrt{3}$, we have for all $z\in (0,\varepsilon^{-1/2}]$ and $y\in(0,\varepsilon^{-1-\kappa}]$,
\begin{align}\label{M'equation}
&Q^{y,(4/\varepsilon)^{2+\kappa},z}\Big(M'\leq C\sqrt{\varepsilon}\Big)\nonumber\\
&\hspace{0.2in}\geq P\bigg(\sup_{0\leq r\leq 1}\bigg(\Big(\frac{\varepsilon}{4}\Big)^{1+\kappa/2}z+\Big(\frac{\varepsilon}{4}\Big)^{1+\kappa/2}yr+\Big|B_{(1)}^{0,1,0}(r)\Big|- \frac{4^{1+\kappa/2}}{\sqrt{3}\rho}\varepsilon^{-\kappa/2}r\bigg)\leq \frac{C}{\sqrt{3}\cdot4^{1+\kappa/2}}\varepsilon^{(\kappa+1)/2} \bigg)\nonumber\\
&\hspace{0.4in} \times \bigg[P\bigg(\sup_{0\leq r\leq 1}\bigg(\Big|B_{(2)}^{0,1,0}(r)\Big|- \frac{4^{1+\kappa/2}}{\sqrt{3}\rho}\varepsilon^{-\kappa/2}r\bigg)\leq \frac{C}{\sqrt{3}\cdot4^{1+\kappa/2}}\varepsilon^{(\kappa+1)/2} \bigg)\bigg]^2\nonumber\\
&\hspace{0.2in}\geq P\bigg(\sup_{0\leq r\leq 1}\bigg(\Big|B_{(1)}^{0,1,0}(r)\Big|- \frac{1}{2} \varepsilon^{-\kappa/2}r\bigg)\leq \frac{1}{4^{1+\kappa/2}}\varepsilon^{(\kappa+1)/2}\bigg)\nonumber\\
&\hspace{0.4in}\times \bigg[P\bigg(\sup_{0\leq r\leq 1}\bigg(\Big|B_{(2)}^{0,1,0}(r)\Big|- \frac{4^{1+\kappa/2}}{\sqrt{3}\rho}\varepsilon^{-\kappa/2}r\bigg) \leq \frac{C}{\sqrt{3}\cdot4^{1+\kappa/2}}\varepsilon^{(\kappa+1)/2}\bigg)\bigg]^2.
\end{align}
According to Lemma \ref{RBMa=0}, for $\varepsilon$ sufficiently small, we have
\begin{equation}\label{M'estimate1}
P\bigg(\sup_{0\leq r\leq 1}\bigg(\Big|B_{(1)}^{0,1,0}(r)\Big|- \frac{1}{2} \varepsilon^{-\kappa/2}r\bigg)\leq \frac{1}{4^{1+\kappa/2}}\varepsilon^{(\kappa+1)/2}\bigg)\geq \varepsilon^{-1/4}\exp\bigg\{-\frac{4\pi^2}{\sqrt{\varepsilon}}\bigg\},
\end{equation} 
\begin{equation}\label{M'estimate2}
P\bigg(\sup_{0\leq r\leq 1}\bigg(\Big|B_{(2)}^{0,1,0}(r)\Big|- \frac{4^{1+\kappa/2}}{\sqrt{3}\rho}\varepsilon^{-\kappa/2}r\bigg) \leq \frac{C}{\sqrt{3}\cdot4^{1+\kappa/2}}\varepsilon^{(\kappa+1)/2}\bigg)\geq \sqrt{\frac{3\rho\pi}{C}}\varepsilon^{-1/4}\exp\bigg\{-\frac{3\rho\pi^2}{8C\sqrt{\varepsilon}}\bigg\}.
\end{equation}
Letting $C_{12}>4\pi^2+12\pi^2/(8C)\geq4\pi^2+6\rho\pi^2/(8C)$, the Lemma follows from equations (\ref{M'equation}), (\ref{M'estimate1}) and (\ref{M'estimate2}).
\qedwhite

\bigskip
\noindent {\bf {\Large Acknowledgment}}
  The author would like to express deep gratitude to Professor Jason Schweinsberg for his patient guidance, constructive suggestions and useful critiques during the planning and development of this paper. The author would also like to thank two referees for carefully reading the paper and providing helpful comments. The author's research was supported in part by NSF Grant DMS-1707953.

\end{document}